\documentclass[12pt]{amsart}

\usepackage{amsmath,amsfonts,bbm,amsthm}
\usepackage[latin1]{inputenc}  
\usepackage[T1]{fontenc}       
\usepackage[psamsfonts]{amssymb}
\usepackage{hyperref}
\usepackage{color}
\usepackage{enumerate}

\renewcommand{\theequation}{\thesection.\arabic{equation}}
\newtheorem{theorem}{Theorem}
\newtheorem{lemma}{Lemma}
\newtheorem{proposition}{Proposition}
\newtheorem{corollary}{Corollary}
\newtheorem{remark}{Remark}
\newtheorem{definition}{Definition}

\newtheorem{property}{Property}

\usepackage{geometry}
\usepackage{tikz}

\newcommand{\eqnsection}{
\renewcommand{\theequation}{\thesection.\arabic{equation}}
    \makeatletter
    \csname  @addtoreset\endcsname{equation}{section}
    \makeatother}
\eqnsection

\def\d{\, \mathrm{d}}

\def\w{\omega}
\def\W{\Omega}

\def\ccc{{\cal C}}

\def\Z{{{\Bbb Z}}}
\def\P{{{\Bbb P}}}
\def\N{{{\Bbb N}}}
\def\R{{{\Bbb R}}}

\def\ttt{{\mathcal T}}
\def\ccc{{\mathcal C}}
\def\ue{{\underline e}}
\def\oe{{\overline e}}
\def\E{{{\Bbb E}}}

\def\hhh{{{\mathcal H}}}

\def\indic{{{\mathbbm 1}}}

\DeclareMathOperator{\strength}{strength}
\def\Q{{\mathbb Q}}
\def\equilaw{{\stackrel{\rm law}{=}}}

\newcommand{\limites}[2]{\overset{#1}{\underset{#2}{\longrightarrow}}}

\newcommand{\equivalent}[1]{\overset{}{\underset{#1}{\sim}}}

\newcommand{\alphach}{{\check{\alpha}}}

\newcommand{\omegach}{{\check{\omega}}}

\newcommand{\eps}{\varepsilon}
\DeclareMathOperator{\diverg}{div}
\newcommand{\dive}{\diverg}
\newcommand{\dir}{\mathcal D}
\DeclareMathOperator{\Beta}{Beta}
\usepackage{graphicx}
\newcommand{\figsection}{
\renewcommand{\thefigure}{\thesection.\arabic{figure}}
    \makeatletter
    \csname  @addtoreset\endcsname{figure}{section}
    \makeatother}
\figsection

\newcommand {\Sum} {\sum\limits}

\author[C. Sabot]{Christophe Sabot}
\address{
Université de Lyon, Université Lyon 1, CNRS UMR5208, Institut
Camille Jordan, 43, bd du 11 nov., 69622 Villeurbanne cedex}
\email{sabot@math.univ-lyon1.fr}

\author[L. Tournier]{Laurent Tournier}
\address{
Université Paris 13, Sorbonne Paris Cité, LAGA, CNRS UMR 7539,
93430 Villetaneuse} 
\email{tournier@math.univ-paris13.fr}

\title[Random walks in Dirichlet environment]{Random walks in Dirichlet environment:\\ an overview}

\thanks{This work was supported by the ANR project MEMEMO2}

\begin{document}

\maketitle

\noindent{\bf Abstract:} 
Random Walks in Dirichlet Environment (RWDE) correspond to Random Walks in Random Environment (RWRE) on $\Z^d$
where the transition probabilities are i.i.d.\ at each site with a Dirichlet distribution. Hence, the model is parametrized by a family of positive weights $(\alpha_i)_{i=1, \ldots, 2d}$, one for each direction of $\Z^d$.
In this case, the annealed law is that of a reinforced random walk,
with linear reinforcement on directed edges. RWDE have a remarkable property of
statistical invariance by time reversal from which can be inferred several properties
that are still inaccessible for general environments, such as the equivalence of static and dynamic points of view and a
description of the directionally transient and ballistic regimes.
In this paper we give a state of the art on this model and several sketches of proofs presenting the core of the arguments.
We also present new computation of the large deviation rate function for one dimensional RWDE.
\section{Introduction}

Multidimensional  Random Walks in Random Environment (RWRE) have been the object of intense investigation in the last fifteen years.
Important progress has been made but some central questions remain open.
The ballistic case, i.e.\ the case where an a priori ballistic condition (as the $(T)_\gamma$ condition of Sznitman, \cite{sznitman2001})
is assumed, is by far the best understood (cf e.g. \cite{kalikow1981,sznitman-zerner1999,sznitman-00,sznitman2001,RS-Sepp,berger-ramirez2014}). 
The non ballistic case is more difficult and researches have concentrated on
the perturbative regime, where the environment is assumed to be a small perturbation of the simple random walk (see in particular 
\cite{sznitman-zeitouni2006,bolthausen-zeitouni2007}), or on two special cases,
the balanced case (\cite{Lawler-82}) and the Dirichlet case. 
The object of this paper is to give an overview of what is known in this last case:
Random Walks in Dirichlet Environment (RWDE) correspond to a special instance of i.i.d.\ random environment where the environment
at each site is chosen according to a Dirichlet random variable. Note that compared to the balanced case, where the drift of the environment at each site
is almost surely null, there is no almost sure restriction on the possible environments, more precisely the support of the law on the environment
is the whole set of environments. The main property that justifies the interest in this special case is a property of
statistical invariance by time reversal (cf.\ Section \ref{sec:reversal}) from which several results can be inferred and which is the main focus of this paper.

For simplicity, in this paper we restrict ourselves to the case of RWRE on $\Z^d$ to nearest neighbors, (except for Sections \ref{sec:rwde} and \ref{sec:reversal}),
even if most of the results on RWDE
could be extended to more general settings. Denote $(e_1, \ldots, e_d)$ the canonical basis of $\Z^d$ and set
$e_{i+d}=-e_i$ so that $(e_1, \ldots, e_{2d})$ is the set of unit vectors of $\Z^d$.
Recall that in this case the set of environments is the set
$$
\W=\Big\{(\w(x,x+e_i))_{x\in \Z^d, \; i=1,\ldots, 2d}\,:\,\w(x,x+e_i)\in (0,1),\, \sum_{i=1}^{2d} \w(x,x+e_i)=1\Big\}.
$$
  The classical model of RWRE is the model where the transition probabilities at each site
  $(\w(x,x+\cdot))$  are independent with a same law $\mu$, which is a distribution on the simplex
\begin{eqnarray}\label{simplex}
\Delta_{2d}=\Big\{(\w_i)_{i=1, \ldots, 2d}\in (0,1)^{2d}\,:\,\sum_{i=1}^{2d} \w_i=1\Big\}.
\end{eqnarray}
We denote by $\P=\mu^{\otimes \Z^d}$,  the law obtained on $\W$.
Traditionally, the quenched law, which is the law of the Markov chain $(X_n)_{n\ge 0}$ in a fixed environment $\w$,
is denoted by $P_{x_0,\w}$, i.e.\ we have $P_{x_0,\w}(X_0=x_0)=1$ and
\begin{eqnarray*}
P_{x_0,\w}\left(X_{n+1}=x+e_i \;|\; X_n=x, \;(X_k)_{k\le n}\right)=\w(x,x+e_i).
\end{eqnarray*}
The annealed law is the law obtained after expectation with respect to the environment
$$
P_{x_0}(\cdot)=\int P_{x_0,\w}(\cdot) \; \P(\d\w).
$$

To define Dirichlet environment, we fix some positive parameters $(\alpha_i)_{i=1, \ldots, 2d}$, one for each direction
$(e_i)_{i=1,\ldots, 2d}$ of $\Z^d$.
The RWDE with parameters $(\alpha_i)_{i=1, \ldots, 2d}$ is the RWRE with the following specific choice for $\mu$.
We choose $\mu={\mathcal{D}}((\alpha_i)_{i=1, \ldots, 2d})$, which is the Dirichlet law with parameters $(\alpha_i)$: 
it corresponds to the law on the simplex (\ref{simplex}) with distribution
$$
{\Gamma(\sum_{i=1}^{2d} \alpha_i)\over \prod_{i=1}^{2d}\Gamma(\alpha_i)}
\left(\prod_{i=1}^{2d} \w_i^{\alpha_i-1}\right)\indic_{\Delta_{2d}}(\w) \left(\prod_{i\neq i_0} \d\w_i\right),
$$
where $\Gamma(\alpha)=\int_0^\infty t^{\alpha-1}e^{-t} \d t$ is the usual Gamma function, and $i_0$ is an irrelevant choice
of index in $\{1, \ldots, 2d\}$ (i.e., we integrate on $\prod_{i\neq i_0} \d\w_i$ with $\w_{i_0}=1-\sum_{i\neq i_0} \w_i$).
We denote by $\P^{(\alpha)}$ the associated law on $\W$ and by $P_{x_0}^{(\alpha)}$ the annealed law of RWDE.

The Dirichlet law is a classical law that plays an important role in Bayesian statistics. It is also intimately related to
P\'olya urns (cf.~Section~\ref{Dir-Polya}), and this relation implies that the annealed law of RWDE is that of a directed edge reinforced
random walk (cf.~Section~\ref{dir-reinforced}).
The important property that justifies that RWDE is an interesting special case of RWRE, is the aforementioned property of statistical invariance by time reversal.
It asserts that on finite graphs, under a condition of zero divergence of the weights, the time reversed environment is again a Dirichlet
environment, in particular time reversed transition probabilities are independent at each site. In this paper we review what is known on Dirichlet environments
(with a few new results) and give sketches of proofs or new proofs of some of the results involving the time reversal property.
This property has been principally applied in the following directions:
\begin{itemize}
\item
Description of directionally transient/recurrent regimes in any dimension (implying in particular a positive
answer to directional 0-1 law).
\item
Proof of transience in dimension $d\ge 3$ for all parameters.
\item
Characterization of the parameters for which there is equivalence between static and dynamic points of view
in dimension $d\ge 3$.
\item
Characterization of ballistic regimes in dimension $d\ge 3$, which gives an answer in this context to 
the question of the equivalence between directional transience and ballisticity.
\end{itemize}

Let us also mention, on a different but somehow related model of random walk in space time Beta random environment, the recent work of Barraquand and Corwin,
\cite{Barraquand-Corwin-2015}, where
Tracy-Widom distribution appears in the second order correction in the large deviation principle. This model is closely related to the exactly solvable
model of log-gamma polymers introduced by Sepp\"al\"ainen, \cite{seppalainen2012}.

Let us describe the organization  of the paper. In Section~\ref{sec:rwde}, we give definition and basic properties of Dirichlet laws, P\'olya urns and
RWDE.
In Section~\ref{sec:reversal}, we state the important property of statistical invariance by time reversal. We give a proof, slightly shorter
than that of \cite{sabot-tournier2011}.
In Section~\ref{traps}, we explain the role of traps of finite size and define the important
parameter $\kappa$.
In Section~\ref{ss:results-reversal} we state the main results which are consequences of the time reversal property.
In Section~\ref{sec:CLT}, we consider the question of quenched central limit theorems in the case of ballistic RWDE.
In Section~\ref{sec:proofs}, we give sketches of proofs and some extensions of the results involving the
time reversal property. We do not optimize on the parameters, which makes the proofs more transparent than the originals.
Finally, in Section~\ref{sec:1D}, we describe the case of one-dimensional RWDE, for which special calculations can be made.
In particular, we give an explicit new computation of the rate function of one-dimensional RWDE.

\section{Random Walk in Dirichlet environment and directed edge reinforced random walk}
\label{sec:rwde}

\subsection{Dirichlet laws and P\'olya urns}\label{Dir-Polya}

Dirichlet distributions classically arise as the limit distribution of colors in P\'olya urns. Let us recall this result. 

An urn contains balls of $r$ different colors. Initially, $\alpha_i$ balls of color $i$ are present, for $i=1,\ldots,r$. After each draw, the ball is put back in the urn together with one additional ball of the same color. In other words, if $(X_n)_{n\ge1}$ denotes the sequence of colors drawn from the urn, and $\mathcal F_n=\sigma(X_1,\ldots,X_n)$, then we have for all $n$, for $i=1,\ldots,r$,
\begin{equation}
	P(X_{n+1}=i\,|\,\mathcal F_n)=\frac{N_i(n)}{\alpha_1+\cdots+\alpha_r+n},\label{eq:polya}
\end{equation}
where $N_i(n)=\alpha_i+\#\{1\le k\le n\,:\,X_k=i\}$ is the number of balls of color $i$ in the urn after $n$ draws. Such an urn is usually called \emph{reinforced} as the chosen color becomes more likely in the future draws. Let us underline that the formal definition of the model doesn't require the $\alpha_i$'s to be integers but merely positive real numbers. 

One can check that the proportion of balls of color $i$ after $n$ draws, $M_i(n)=\frac{N_i(n)}{\alpha_1+\cdots+\alpha_r+n}$, is a bounded martingale and therefore converges almost surely to a random variable~$U_i$. Note that the vector $(U_1,\ldots,U_r)$ takes values in the simplex
\[\Delta_r=\big\{(u_1,\ldots,u_r)\in(0,1)^r\,:\,u_1+\cdots+u_r=1\big\}.\]
Before stating the main result about P\'olya urns, let us give a central definition: 

\begin{definition}
Given positive real numbers $\alpha_1,\ldots,\alpha_r$, the Dirichlet distribution with parameters $\alpha_1,\ldots,\alpha_r$, is the distribution on $\Delta_r$ given by
\[\dir(\alpha_1,\ldots,\alpha_r)=\frac{\Gamma(\alpha_1+\cdots+\alpha_r)}{\Gamma(\alpha_1)\cdots\Gamma(\alpha_r)}u_1^{\alpha_1-1}\cdots u_r^{\alpha_r-1}\d\lambda_{\Delta_r}(u_1,\ldots,u_r),\]
where $\d\lambda_{\Delta_r}$ is the Lebesgue measure on $\Delta_r$, that is to say $\indic_{\Delta_r}(u_1,\ldots,u_r)\prod_{i\ne i_0} \d u_i$ for an arbitrary choice of $i_0\in\{1,\ldots,r\}$. 
\end{definition}
A classical proof of the above normalization would consist in writing
\[\Gamma(\alpha_1)\cdots\Gamma(\alpha_r)=\int_{(0,+\infty)^r}x_1^{\alpha_1-1}\cdots x_r^{\alpha_r-1}e^{-(x_1+\cdots+x_r)}\d x_1\cdots\d x_r,\]
and letting $u_i=x_i/(x_1+\cdots+x_r)$ for $i=1,\ldots,r-1$, and $v=x_1+\cdots+x_r$. This is tightly related to Property~\ref{pro:gamma_ratio} below. As a consequence of this normalization, we immediately deduce joints moments of marginals of the Dirichlet distribution: if $(V_1,\ldots,V_r)\sim\dir(\alpha_1,\ldots,\alpha_r)$, then
\begin{equation}
	E\big[(V_1)^{\xi_1}\cdots(V_r)^{\xi_r}\big]=\frac{\Gamma(\alpha_1+\xi_1)\cdots\Gamma(\alpha_r+\xi_r)}{\Gamma(\alpha_1+\xi_1+\cdots+\alpha_r+\xi_r)}\frac{\Gamma(\alpha_1+\cdots+\alpha_r)}{\Gamma(\alpha_1)\cdots\Gamma(\alpha_r)},\label{eq:moments}
\end{equation}
for all real numbers $\xi_1,\ldots,\xi_r\in\R$ such that $\alpha_i+\xi_i>0$ for all $i$. If $\xi_i+\alpha_i\le0$ for some~$i$, then the expectation is infinite. In the case when $\xi_i$'s are integers, the functional equation of the gamma function reduces the previous formula to an elementary product that has an interpretation in terms of P\'olya urn and leads to the next lemma.

\begin{lemma}\label{lem:polya}
	The vector $U=(U_1,\ldots,U_r)$ of asymptotic proportions of colors in the P\'olya urn follows the Dirichlet distribution $\dir(\alpha_1,\ldots,\alpha_r)$. Furthermore, conditional on $U$, the sequence $(X_n)_{n\ge1}$ is independent and identically distributed with, for $n\ge1$ and $i=1,\ldots,r$, 
	\[P(X_n=i\,|\,U_1,\ldots,U_r)=U_i.\]
\end{lemma}

\begin{proof}
	It is a simple matter to check that, for any $x_1,\ldots,x_n\in\{1,\ldots,r\}$, if we let $n_i=\#\{1\le k\le n\,:\,x_k=i\}$ for $i=1,\ldots,r$, 
\begin{align}
P(X_1=x_1,\ldots,X_n=x_n)
	& =\frac{\prod_{i=1}^r \alpha_i(\alpha_i+1)\cdots(\alpha_i+n_i-1)}{(\alpha_1+\cdots+\alpha_r)\cdots(\alpha_1+\cdots+\alpha_r+n-1)}\notag\\
	& = \frac{\Gamma(\alpha_1+n_1)}{\Gamma(\alpha_1)}\cdots\frac{\Gamma(\alpha_r+n_r)}{\Gamma(\alpha_r)}\frac{\Gamma(\alpha_1+\cdots+\alpha_r)}{\Gamma(\alpha_1+\cdots+\alpha_r+n_1+\cdots+n_r)}\notag\\
	& = E[(V_1)^{n_1}\cdots(V_r)^{n_r}],\notag
\end{align}
where $V=(V_1,\ldots,V_r)\sim\dir(\alpha_1,\ldots,\alpha_r)$. Thus, $(X_n)_n$ has same law as i.i.d.~variables $(Y_n)_n$ with common law $V_1\delta_1+\cdots+V_r\delta_r$ given $V$, where $V\sim\dir(\alpha_1,\ldots,\alpha_r)$. By the law of large numbers for the $Y_n$'s given $V$, $V$ is the vector of almost sure limiting proportions of colors in the sequence $(Y_n)_n$, hence $((X_n)_n,U)$ has same law as $((Y_n)_n,V)$, which concludes. Note that this actually re-proves the almost sure convergence of proportions of colors toward $U$ without a martingale convergence theorem. 
\end{proof}

This lemma can also be seen as an instance of de Finetti's theorem (see for instance~\cite[p.\,268]{durrett}), since it is easily noticed that the sequence $(X_n)_n$ is exchangeable.

In the usual two-color case, we have $(U_1,1-U_1)\sim\dir(\alpha_1,\alpha_2)$, which reduces to
\[U_1\sim\Beta(\alpha_1,\alpha_2)=\frac1{B(\alpha_1,\alpha_2)}u^{\alpha_1-1}(1-u)^{\alpha_2-1}\indic_{(0,1)}(u)\d u,\] 
where $B(\alpha,\beta)=\frac{\Gamma(\alpha)\Gamma(\beta)}{\Gamma(\alpha+\beta)}$ is the Beta function. 

For later convenience, we will also consider more general index sets:
\begin{definition}
	For a finite set $I$ and $(\alpha_i)_{i\in I}\in(0,+\infty)^I$, the Dirichlet distribution $\dir((\alpha_i)_{i\in I})$ on
\[\Delta_I=\Big\{(u_i)_{i\in I}\in(0,1)^r\,:\,\sum_{i\in I}u_i=1\Big\}\]
is given by
\[\dir((\alpha_i)_{i\in I})=\frac{\Gamma(\sum_{i\in I}\alpha_i)}{\prod_{i\in I}\Gamma(\alpha_i)}\bigg(\prod_{i\in I}u_i^{\alpha_i-1}\bigg)\indic_{\Delta_I}((u_i)_{i\in I})\prod_{i\ne i_0} \d u_i,\]
for an irrelevant choice of $i_0$. 
\end{definition}

NB.~From Lemma~\ref{lem:polya} for instance, or continuity in distribution, it is natural to allow some parameters of the distribution, but not all, to be zero, by setting these coordinates equal to $0$ a.s.~and viewing $\dir(\alpha)$ as a distribution on $\Delta_{\{i\,:\,\alpha_i\ne 0\}}$.

\subsection{Properties of Dirichlet distributions}
\label{sec:dirichlet}
Let $I$ be finite, and $(\alpha_i)_{i\in I}\in (0,+\infty)^I$. 

Dirichlet distribution could equivalently have been defined as the law of a normalized Gamma vector. By routine computation, one can indeed check that:

\begin{property}\label{pro:gamma_ratio}
	Let $(W_i)_{i\in I}$ be independent random variables such that, for $i\in I$, \[W_i\sim\Gamma(\alpha_i,1)=\frac1{\Gamma(\alpha_i)}w^{\alpha_i-1}e^{-w}\indic_{(0,\infty)}(w)\d w.\] We have
	\[(U_i)_{i\in I}:=\frac1{\sum_{i\in I}W_i}\big(W_i\big)_{i\in I}\sim\dir((\alpha_i)_{i\in I}),\]
	and $(U_i)_{i\in I}$ is independent of $\sum_{i\in I} W_i$.
\end{property}

Recall that, if $X\sim\Gamma(\alpha,1)$ and $Y\sim\Gamma(\beta,1)$ are independent, then $X+Y\sim\Gamma(\alpha+\beta,1)$. Together with the previous property, this gives:

\begin{property}\label{pro:properties}
Assume $(U_i)_{i\in I}$ has Dirichlet distribution $\dir((\alpha_i)_{i\in I})$. 
\begin{description}
	\item[(Agglomeration)] Let $I_1,\ldots,I_n$ be a partition of $I$. The random variable $\left(\sum_{i\in I_k}U_i \right)_{k\in\{1,\ldots,n\}}$ on $\Delta_n$ follows the Dirichlet distribution $\dir((\sum_{i\in I_k} \alpha_i)_{1\leq k\leq n})$. 
	\item[(Restriction)] Let $J$ be a nonempty subset of $I$. The random variable $\left(\frac{U_i}{\sum_{j\in J} U_j}\right)_{i\in J}$ on $\Delta_J$ follows the Dirichlet distribution $\dir((\alpha_i)_{i\in J})$ and is independent of $\sum_{j\in J}U_j$. 
\end{description}
\end{property}

In particular, from the agglomeration property, the marginal $U_i$ of a Dirichlet vector $(U_i)_{i\in I}\sim\dir((\alpha_i)_{i\in I})$ follows the law $\Beta(\alpha_i,\sum_{j\ne i}\alpha_j)$. 

One may notice that Property~\ref{pro:properties} can also be elementarily deduced from Lemma~\ref{lem:polya} by identifying together or disregarding some colors. 

Property~\ref{pro:gamma_ratio} also enables to derive the following degenerate large weights limit, which means that the effect of reinforcement vanishes as the initial number of balls goes to infinity: 
\begin{equation}
	\dir((\lambda\cdot\alpha_i)_{i\in I})\limites{}{\lambda\to\infty}\delta_{\frac1{\sum_i \alpha_i}(\alpha_i)_i}.\label{eq:large-weights}
\end{equation}

In the opposite direction, with small weights, the distribution concentrates on the extreme points of the simplex, which means that the first draw from the urn becomes ``overwhelming'' as the initial weights go to 0: with $(\indic_{\{i\}})_j=\delta_{ij}$ for $i,j\in I$, 
\begin{equation}
	\dir((\lambda\cdot\alpha_i)_{i\in I})\limites{}{\lambda\to0^+}\sum_{i\in I}\frac{\alpha_i}{\sum_j \alpha_j}\delta_{\indic_{\{i\}}}.\label{eq:small-weights}
\end{equation}

(These asymptotics can be quickly obtained by taking the limit in the joint moments given in the proof of Lemma~\ref{lem:polya})

\subsection{RWDE on general graphs, and reinforcement}\label{dir-reinforced}

Let $G=(V,E)$ be a locally finite directed graph. Recall that \emph{directed} means that edges $e=(x,y)\in E\subset V\times V$ have a tail $\underline{e}=x$ and a head $\overline{e}=y$, while \emph{locally finite} means that vertices have finite degree. 

Let $\alpha=(\alpha_e)_{e\in E}\in(0,+\infty)^E$ be positive weights on the edges. 

We denote by $(X_n)_{n\ge0}$ the canonical process on $V$. 

\begin{definition}
	Let $x_0\in V$. The \emph{directed edge linearly reinforced random walk on $G$ with initial weights $\alpha$ and starting at $x_0$} is the process on $V$ with law $P^{(\alpha)}_{x_0}$ defined by: $P^{(\alpha)}_{x_0}$-a.s., $X_0=x_0$ and, for all $n\ge0$, for all edges $e\in E$,
	\[P^{(\alpha)}_{x_0}((X_n,X_{n+1})=e\,|\,X_0,\ldots,X_n)=\frac{N_e{(n)}}{\sum_{f\in E,\,\underline{f}=\underline e}N_f{(n)}}\indic_{\{\underline e=X_n\}},\]
	where $N_e{(n)}=\alpha_e+\#\{0\le k\le n-1\,:\,(X_k,X_{k+1})=e\}$.
\end{definition}

In other words, at time $n$, this walk jumps through a neighboring edge $e$ chosen with probability proportional to its current weight $N_e(n)$, where this weight initially was equal to $\alpha_e$ and then increased by $1$ each time the edge $e$ was chosen. 

Since edges are oriented, the decisions of this process are ruled by independent P\'olya urns, one per vertex, where outgoing edges play the role of colors, and $\alpha$ is the initial numbers of balls of each color. By Lemma~\ref{lem:polya}, this reinforced walk may equivalently be obtained by assigning a Dirichlet random variable $\omega_{(x,\cdot)}$ to each vertex $x$, and sampling i.i.d.~edges according to this variable in order to define the next step of the walk every time it is at $x$: this is the description of a \emph{random walk in Dirichlet random environment} (RWDE), that we formalize now. 

The set of \emph{environments} on $G$ is
\[\Omega_G=\prod_{x\in E}\Delta_{\{e\in E\,:\,\underline e=x\}}=\Big\{(\w_e)_{e\in E}\in(0,1]^E\,:\,\mbox{for all } x\in V, \sum_{e\in E,\,\underline{e}=x}\w_e=1\Big\},\]
and we shall denote by $\omega$ the canonical random variable on $\Omega_G$. 

\begin{definition}
	Let $x_0\in V$. For $\omega\in\Omega_G$, the \emph{quenched random walk in environment~$\omega$ starting at~$x_0$} is the Markov chain on $V$ starting at $x_0$ and with transition probabilities~$\omega$. We denote its law by $P_{x_0,\omega}$. Thus, $P_{x_0,\omega}$-a.s., $X_0=x_0$ and for all $n$, for all $e\in E$,
	\[P_{x_0,\omega}\big((X_n,X_{n+1})=e\,\big|\,X_0,\ldots,X_n\big)=\omega_e \indic_{\{\underline e=X_n\}}.\]

	The \emph{Dirichlet distribution on $G$ with parameter $\alpha$} is the product distribution on $\Omega_G$
	\[\P^{(\alpha)}=\prod_{x\in V}\dir((\alpha_e)_{\{e\in E\,:\,\underline e=x\}}).\]
	Thus, under $\P^{(\alpha)}$, the random variables $\omega_{(x,\cdot)}$, $x\in V$, are independent and follow Dirichlet distributions with parameters given by $\alpha_{(x,\cdot)}$, $x\in V$, respectively.

	Let us consider the joint law $P^{(\alpha)}_{x_0}$ of $(\omega,X)$ on $\Omega_G\times V^\N$ such that $\omega\sim\P^{(\alpha)}$ and the conditional distribution of $X$ given $\omega$ is $P_{x_0,\omega}$. Then, under $P^{(\alpha)}_{x_0}$, $X$ is the \emph{annealed random walk in Dirichlet environment with parameter $\alpha$ starting at $x_0$}. Its law is thus
	\[P_{x_0}^{(\alpha)}(X\in\cdot)=\E^{(\alpha)}[P_{x_0,\omega}(\cdot)].\]

\end{definition}

Because of the previous remark, it follows from Lemma~\ref{lem:polya} that the notation $P^{(\alpha)}_{x_0}$ is unconsequently ambigous: 

\begin{lemma}\label{lem:reinforcedrwre}
	Let $x_0\in V$. The directed edge linearly reinforced random walk on $G$ with initial weights $\alpha$ starting at $x_0$, and the annealed random walk in Dirichlet environment with parameter $\alpha$ starting at $x_0$, are equal in distribution. 
\end{lemma}

Given the natural definition of directed edge reinforced random walk, this property provides a first justification for the interest in RWDE. 

Let us mention that this connection was first used in the context of (non oriented) edge reinforced random walks on trees by Pemantle~\cite{pemantle1988} where, due to the absence of cycles, independence between the P\'olya urns still holds. On other graphs, non oriented edge reinforced random walks can be seen as random walks in a correlated, yet rather explicit, random environment. This leads to very different behaviors and techniques, see for instance~\cite{coppersmith,Keane-Rolles,Merkl-Rolles,sabot-tarres2011,angel-crawford-kozma}.
RWDE were first considered for their own as a special instance of RWRE in $\Z^d$ by Enriquez and Sabot, \cite{enriquez-sabot2006}.

For any vertex $x$, we let $\alpha_x$ be the sum of the weights of the edges exiting from $x$: \[\alpha_x=\sum_{e\in E,\,\underline{e}=x} \alpha_e.\] 
With this notation, when $G$ is finite, the Dirichlet distribution may be written as
\begin{equation}
\dir(\alpha)=\frac1{Z_\alpha}\prod_{e\in E}\w_e^{\alpha_e-1}\prod_{e\in\widetilde{E}}\d \w_e,\qquad\text{where}\quad
Z_\alpha=\frac{\prod_{e\in E}\Gamma(\alpha_e)}{\prod_{x\in V}\Gamma(\alpha_x)},\label{eq:defZ}
\end{equation}
and $\widetilde{E}$ is obtained from $E$ by removing arbitrarily, for each $x\in V$, one edge with origin~$x$. 
From this we can infer the following formula for the moments of $\omega$:
\begin{eqnarray}\label{moments}
	\E^{(\alpha)}\bigg[ \prod_{e\in E} \omega_e^{\xi_e}\bigg] = {Z_{\alpha+\xi}\over Z_\alpha}=\left(\prod_{e\in E} {\Gamma(\alpha_e+\xi_e)\over \Gamma(\alpha_e)}\right)
\left(\prod_{x\in V} {\Gamma(\alpha_x)\over \Gamma(\alpha_x+\xi_x)}\right),
\end{eqnarray}
for every function $(\xi_e)\in \R^E$ such that $\alpha_e+\xi_e>0$ for all $e$, and where as usual we write $\xi_x=\sum_{e, \underline e =x} \xi_e$.
When $\xi_e+\alpha_e\le0$ for some edge $e$, the expectation is infinite.

In the following, we are mainly interested in the case of $\Z^d$ with nearest-neighbor edges. There we always assume that weights are translation invariant, therefore given by $2d$ parameters $\alpha_1,\ldots,\alpha_{2d}$, so that for any $x\in\Z^d$, and $i=1,\ldots,2d$,
\[\alpha_{(x,x+e_i)}=\alpha_i,\]
where $(e_1,\ldots,e_d)$ is the canonical basis of $\Z^d$ and we let $(e_{1+d},\ldots,e_{2d})=-(e_1,\ldots,e_d)$. See figure~\ref{fig:Zd}. 

\begin{figure}
\begin{center}
\includegraphics[height=2.5cm]{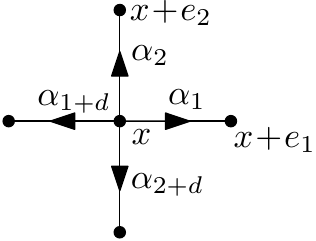}
\caption{Weights on the edges starting at~$x\in\Z^d$}
\label{fig:Zd}
\end{center}
\end{figure}

\section{The property of statistical invariance by time reversal}\label{sec:reversal}
\subsection{The main lemma and a probabilistic proof}
Consider a directed graph $G=(V,E)$ with a family of positive weights $(\alpha_e)_{e\in E}$.
Assume that $G$ is finite and strongly connected, i.e.\ that for any $x$ and $y$, there is a directed path
in $G$ from $x$ to $y$. Consider the dual graph $\check G=(V, \check E)$ obtained by reversing all edges, i.e.
$$
\check E=\{(y,x)\,:\, (x,y)\in E\}.
$$ 
For an edge $e\in E$ we denote $\check e\in \check E$ the associated reversed edge.
We define the family of reversed weights $(\check\alpha_e)_{e\in \check E}$ by
$$
\check\alpha_{\check e}=\alpha_{e}, \;\;\; \forall e\in E.
$$

We define the divergence operator on the graph $G$ as the linear operator $\dive:\R^E\to \R^V$
given by
\begin{eqnarray}\label{divergence}
\dive(\theta)(x)=\sum_{e,\,\underline e=x}\theta_e-\sum_{e,\,\overline e=x } \theta_e, \;\;\; \forall \theta \in \R^E, \;\forall x\in V.
\end{eqnarray}

Let $(\w_e)_{e\in E}$ be an environment on the graph $G$. Since $G$ is finite and strongly connected, there exists
an invariant probability for the quenched Markov chain $P^\w(\cdot)$. Denote it by $(\pi^\w(x))_{x\in V}$.
We define the time reversed environment $(\check \w_e)_{e\in \check E}$, which is an environment on the dual graph
$\check G$, by
$$
\check\w_{y,x}={\pi^\w(x)\over \pi^\w(y)} \w_{x,y}, \;\;\; \forall (x,y)\in E.
$$
The following lemma was stated in \cite{sabot2011transience}, Lemma~1, where it was first given an analytic proof that will be discussed in Subsection~\ref{subsec:analytic} below. A much shorter probabilistic proof was given in \cite{sabot-tournier2011}.
\begin{lemma}
\label{lem:reversal}
Assume $G$ is finite and $\diverg\alpha=0$. Then 
\begin{equation}
\left( \w\sim\P^{(\alpha)}\right) \Rightarrow\left( \omegach\sim\P^{(\alphach)}\right).
\end{equation}
\end{lemma}
\begin{proof}
We give here a proof in the spirit of that of \cite{sabot-tournier2011}, but even shorter. We say that $\sigma=(x_0,x_1,\ldots,x_{n-1},x_n)$ is a directed path if 
 $(x_i,x_{i+1})\in E$ for all $i=0, \ldots, n-1$. It is a directed cycle if moreover $x_n=x_0$.
 For a directed path we set
 \begin{eqnarray}
 \label{wsigma}
  \w_\sigma= \prod_{i=0}^{n-1} \w_{x_i, x_{i+1}}.
 \end{eqnarray}
 If $\sigma$ is a path we write
 $\check\sigma= (x_n, \ldots, x_0)$ the reversed path, which is a directed path in the dual graph $\check G$.
 If $\sigma$ is a directed cycle, we clearly have
 \begin{eqnarray}\label{checkw}
  \omega_\sigma=\omegach_{\check\sigma}.
  \end{eqnarray}
For a cycle $\sigma=(x_0,x_1,\ldots,x_{n-1},x_n=x_0)$
we write
$$N_\sigma(e)= \sum_{i=0}^{n-1} \indic_{\{(x_i,x_{i+1})=e\}}, \;\;\; N_\sigma(x)= \sum_{i=0}^{n-1} \indic_{\{x_i=x\}},
$$
the number of visits of the directed edge $e$ and of the vertex $x$ by the cycle $\sigma$.
If $\mathcal C$ is a finite family of cycles, we write
$$
N_\ccc(e)=\sum_{\sigma\in \ccc}N_\sigma(e), \;\;\; N_\ccc(x)=\sum_{\sigma\in \ccc}N_\sigma(x).
$$
We also write $\check\ccc=\{\check\sigma\,:\,\sigma\in \ccc\}$ for the family of reversed cycles.
We clearly have from (\ref{moments})
$$
\E^{(\alpha)}\bigg[ \prod_{\sigma \in \ccc} \w_\sigma\bigg]=
\left({ \prod_{e\in E} \Gamma(\alpha_e+N_\ccc(e))\over  \prod_{e\in E} \Gamma(\alpha_e)}\right)
\left({ \prod_{x\in E}  \Gamma(\alpha_x) \over
\prod_{x\in E} \Gamma(\alpha_x+N_\ccc(x))}\right).
$$
where we write $\alpha_x=\sum_{e, \; \underline e=x} \alpha_e$. The property $\dive(\alpha)=0$ is equivalent to the fact that
$\alpha_x=\check\alpha_x$ for all vertex $x$. Remark now that for a cycle $N_\sigma(e)=N_{\check\sigma}(\check e)$, and
$N_\sigma(x)=N_{\check\sigma}(x)$ for all edge $e$ and vertex $x$.
Hence from (\ref{checkw}) and the previous remarks, changing $e$ to reversed edges $\check e$, we get
\begin{eqnarray*}
\E^{(\alpha)}\bigg[ \prod_{\sigma \in \check\ccc} \check\w_{\sigma} \bigg]
&=&
\left({ \prod_{e\in \check E} \Gamma(\check\alpha_e+N_{\check\ccc}(e))\over  \prod_{e\in \check E} \Gamma(\check \alpha_e)}\right)
\left({ \prod_{x\in E}  \Gamma(\check \alpha_x) \over
\prod_{x\in E} \Gamma(\check \alpha_x+N_{\check\ccc}(x))}\right)
\\
&=&
\E^{(\check \alpha)}\bigg[ \prod_{\sigma \in \check\ccc} \w_{\sigma} \bigg].
\end{eqnarray*}
By considering concatenations of cycles with themselves, this exactly means that under $\P^{(\alpha)}$ all joint moments of cycles of $\check\w$ coincide with the moments of cycles
under $\P^{(\check\alpha)}$. It implies that the law of  $(\check \w_{ \sigma})_{\sigma\text{ cycle of $\check G$}} $ under $\P^{(\alpha)}$ coincides
with the law $(\w_{\sigma})_{\sigma\text{ cycle of $\check G$}} $ under $\P^{(\check\alpha)}$.
But the law of cycle probabilities determine the law of the Markov chain. Indeed, since $G$ is finite and strongly connected, it is recurrent and thus
$$
\w_{(x,y)}=\sum_\sigma \w_{\sigma},
$$
where the sum runs on the cycles that start by the edge $(x,y)$ and come back only once to $x$. Hence 
$\check\w$ under $\P^{(\alpha)}$ has distribution $\P^{(\check\alpha)}$. 
\end{proof}

We will use several times the following corollary  of this lemma. Assume that the graph is finite and $\dive(\alpha)=0$, and
let $x\in V$ be a specified vertex, and $(y,x)\in E$.
Then, under $\P^{(\alpha)}$,
\begin{eqnarray}\label{returnproba}
P_{x,\w}\left( X_n \hbox{ comes back to $x$ by the edge $(y,x)$}\right)\sim \hbox{Beta}(\alpha_{(y,x)}, \alpha_x-\alpha_{(y,x)}).
\end{eqnarray}
Indeed, it comes from the fact that the left hand side term is the sum
$
\sum_{\sigma} \w_\sigma
 $
 where the sum runs on all cycles starting from $x$ and coming back only once to $x$ and by the edge $(y,x)$.
 By (\ref{checkw}), it equals $\sum_\sigma \check \w_{\check \sigma} = \check \w_{(x,y)}$, by Markov property.
 But $\check\w$ is distributed according to the Dirichlet environment $\P^{(\check \alpha)}$, which implies (\ref{returnproba})
 by the agglomeration property of Dirichlet distributions, cf.~Property~\ref{pro:properties}.

\subsection{Analytic approach.}\label{subsec:analytic} (This section is not necessary in the sequel and can be skipped in first reading.)
The original proof
was analytic and based on a change of variable (published only in the arXiv version of \cite{sabot2011transience}). It is rather technical but gives extra information
on the distribution of the occupation measure of the RWDE.
We only give here the statement, the proof is available in the appendix of \cite{sabot2011transience} (arXiv version).
Let $e_0$ be a specified edge of the graph. Let $\hhh_{e_0}$ be the affine space defined by
\[
\hhh_{e_0}=\{ (z_e)_{e\in E}\in \R^E\,:\,z_{e_0}=1, \;\;
\dive(z)\equiv 0\}.
\]
and $\tilde\Delta_{e_0}$ be the set defined by
\[
\tilde\Delta_{e_0}=\hhh_{e_0}\cap(\R_+^*)^E.
\]
We define
$$
Z_e={\pi^\w_{\ue}\w_e\over
\pi^\w_{\ue_0}\w_{e_0}},
$$
the occupation measure of the edges of the graph, normalized so that $Z_{e_0}=1$. Clearly $(Z_e)_{e\in E}\in \tilde\Delta_{e_0}$.
In the stationary regime, it is proportional to the expected number of traversals of the edge $e$.
The proof was based on the explicit computation of the distribution  of the random variable $(Z_e)_{e\in E}$ under $\P^{(\alpha)}$. 

 Let $T$ be a spanning
tree of the graph $G$ such that $e_0\notin T$. (This is possible
since the graph is strongly connected and thus $e_0$ belongs to at
least one directed cycle of the graph.) We denote
$B=T\cup\{e_0\}$. Then
$(z\mapsto z_e)_{e\in B^c}$ is a dual basis of $\hhh_{e_0}$, it defines a natural measure
on $\tilde\Delta_{e_0}$,
$$
\d\lambda_{\tilde\Delta_{e_0}}=\prod_{e\in B^c} \d z_e,
$$
which does not depend on the choice of $T$, $e_0$. 
Let $x_0\in V$ be
any vertex, and denote $\ttt_{x_0}$ the set of directed spanning trees of the graph 
directed towards the vertex $x_0$.

\begin{lemma}\label{chgt-variables}
Under $\P^{(\alpha)}$, the random variable $(Z_e)_{e\in E}$ has the following distribution on~$\tilde\Delta_{e_0}$:
$$
\left({ \prod_{x\in V} \Gamma(\alpha_x) \over \prod_{e\in E} \Gamma(\alpha_e)}\right)\left( {\prod_{e\in E}
z_e^{\alpha_e-1} \over \prod_{x\in V} z_x^{\alpha_x}}\right)
\left(\sum_{T\in \ttt_{x_0}} \prod_{e\in T} z_e\right)
\d\lambda_{\tilde\Delta_{e_0}},
$$
where as usual $z_x=\sum_{e,\,\ue=x} z_e$
\end{lemma}
\begin{remark}
We can remark that this formula
is reminiscent of the distribution discovered by Diaconis and
Coppersmith
(\cite{coppersmith,Keane-Rolles})
which expresses edge-reinforced random walk as a mixture of
reversible Markov chains. Note that the sum on spanning trees can also be expressed
as a principal minor of the matrix with diagonal coefficients equal to $z_x$ and off diagonal coefficients
equal to $-z_{(x,y)}$. It does not depend on the choice of~$x_0$.
\end{remark}

We see that lemma \ref{lem:reversal} is a direct consequence of the
previous result. Indeed we see that lemma \ref{chgt-variables}
applied to the reversed graph $(\check G, \check E)$, starting
with the weights $\check\alpha_{\check e}=\alpha_e$ gives the same
integrand with $\alpha_x$ replaced by $\check \alpha_x=\sum_{e,
\oe=x}\alpha_e$. The two coincide after the change of variables
exactly when $\dive(\alpha)\equiv 0$.

\section{Preliminaries: traps of finite size. The parameter $\kappa$}\label{traps}

Dirichlet environments are \emph{elliptic}, in the sense that
\begin{equation}
	\P^{(\alpha)}\text{-a.s.,}\qquad\forall e\in E,\ \omega_e>0.\label{eq:elliptic}
\end{equation}
They do not however fulfill the common assumption of \emph{uniform ellipticity}, which means that a uniform positive lower bound $\omega_e\ge\eps>0$ would hold in~\eqref{eq:elliptic}.

Non uniform ellipticity of the environment can create traps of finite size, i.e.\ finite subsets in which the random walk may spend an atypically large time. The strength of a possible trap $A\subset V$ can be measured by the order of the tail of the distribution of the quenched Green function of the walk in $A$. 

Let us first consider the case of a subset consisting in a pair of neighbor vertices, i.e.~$A=\{x,y\}$ such that $(x,y),(y,x)\in E$. In the environment~$\omega$, starting at $x$, the number of visits to $x$ before quitting $A$ is geometric with parameter $1-\omega_{(x,y)}\omega_{(y,x)}$. Therefore, the Green function of the walk in $A$ satisfies
\[G_\omega^{\{x,y\}}(x,x)=\frac1{1-\omega_{(x,y)}\omega_{(y,x)}},\]
hence, since $\omega_{(x,y)}\sim\Beta(\alpha_{(x,y)},\alpha_x-\alpha_{(x,y)})$, $\omega_{(y,x)}\sim\Beta(\alpha_{(y,x)},\alpha_y-\alpha_{(y,x)})$ and these variables are independent, it is a simple check that
\[\E^{(\alpha)}\big[G_\omega^{\{x,y\}}(x,x)^s\big]<\infty\qquad\Leftrightarrow\qquad s<\alpha_{\{x,y\}}:=\alpha_x+\alpha_y-\alpha_{(x,y)}-\alpha_{(y,x)}.\]
In this case, the integrability exponent is the total weight of the edges going out of $A$.

This result extends to any finite subset $A$, as proved by Tournier in~\cite{tournier2009integrability}, in that the integrability exponent of $G_\omega^A(x,x)$, for $x\in A$, is given by the minimum total weight of outgoing edges among the (edge-)subsets of $A$ containing $x$. Let us only give a precise statement in the case of $\Z^d$ with translation invariant weights, where this simplifies: 

\begin{proposition} \label{thm:finite_trap}
	We consider the RWDE on $\Z^d$ with parameters $\alpha_1,\ldots,\alpha_{2d}$. 

Let $A$ be a finite connected subset of $\Z^d$ containing $0$. We have:
\[\E^{(\alpha)}\big[G_\omega^A(0,0)^s\big]<\infty\qquad\Leftrightarrow\qquad s<\min_{x,y\in A,\,|x-y|=1}\alpha_{\{x,y\}}.\]
\end{proposition}

Therefore, the strongest finite traps in $\Z^d$ are actually pairs of vertices, and
\begin{equation}
	\Big(\E^{(\alpha)}\big[G_\omega^A(0,0)^s\big]<\infty\text{ for all finite } A\subset\Z^d\text{ containing $0$}\Big)\qquad\Leftrightarrow\qquad s<\kappa, \label{eq:finite_traps}
\end{equation}
where
\begin{equation}
	\fbox{$\displaystyle \kappa = \min_{1\le i\le d}\alpha_{\{0,e_i\}} = 2\sum_{i=1}^{2d}\alpha_i-\max_{1\le i\le d}(\alpha_i+\alpha_{i+d})$}.\label{eq:def_kappa}
\end{equation}

This parameter $\kappa$ plays a central role in forthcoming results. Let us first mention that, if $\kappa\le1$ then the expected exit time out of some pair $\{x,y\}$ under $P^{(\alpha)}_x$ is infinite, which quickly implies non-ballisticity (see~\cite[Proposition 12]{tournier2009integrability} for details):

\begin{proposition}\label{pro:traps}
	If $\kappa\le1$, then $P^{(\alpha)}_0$-a.s., $\displaystyle\lim_{n\to\infty}\frac{X_n}n=0$.
\end{proposition}

As we will see below (Theorem~\ref{ballisticity}), the assumption $\kappa\le1$ is sharp in dimension $d\ge3$, and this is also conjectured to be true in dimension 2. However, in the one-dimensional case, nontrivial traps of all sizes spontaneously appear (in the form of ``valleys'' of the potential) even for large values of the parameters, which leads to the definition of a different threshold $\kappa_1(=\alpha_1-\alpha_2)$, see Section~\ref{sec:1D}.

\section{RWDE on $\Z^d$, $d\ge 2$. Results involving the time reversal property}
\label{ss:results-reversal}
In this section we describe the consequences of the time reversal property described in the previous section.
This property has been successfully applied in three directions: directional transience (\cite{sabot-tournier2011, tournier2015direction}),
transience in dimension $d\ge 3$ (\cite{sabot2011transience}), invariant measure viewed from the particule (\cite{sabot2013particle,bouchet2013}).
These results have consequences on ballisticity conditions, in particular the question of equivalence between ballisticity and
directional transience, directional 0-1 law.
When it can be applied, this property in general provides optimal conditions, and gives information that are not accessible
for general environments. 

Sketches for the proofs of the results in this section are given in Section~\ref{sec:proofs}.

\subsection{Directional transience}
\label{sub:transience_dir}

The walk $(X_n)_n$ is said to be \emph{transient in direction $\ell$} if $X_n\cdot\ell\to+\infty$. 
The question of directional transience or recurrence is now completely understood in the case of Dirichlet environment.

Let us denote
\[
d_\alpha= \sum_{i=1}^{2d} \alpha_i e_i.
\]
In particular, $E^{(\alpha)}_0[X_1]=\frac1{\alpha_1+\cdots+\alpha_{2d}}d_\alpha$ is the drift of the mean environment. 

The main result is the following. It follows from~\cite{bouchet2013} and~\cite{tournier2015direction}. 

\begin{theorem}
\label{th:transience_dir}
Let $\ell\in \R^d\setminus\{0\}$.
\begin{enumerate}[(i)]
	\item If $d_\alpha \cdot \ell =0$, then
\begin{equation*}
	P^{(\alpha)}_0\text{-a.s.,}\qquad -\infty=\liminf_{n\to\infty} X_n\cdot \ell<\limsup_{n\to\infty} X_n\cdot \ell=+\infty. 
\end{equation*}
	\item If $d_\alpha \cdot \ell >0$ (resp. $d_\alpha \cdot \ell < 0$)
\begin{equation*}
	P^{(\alpha)}_0\text{-a.s.,}\qquad \lim_{n\to\infty}X_n\cdot \ell+\infty \;\;\; (\hbox{resp. $-\infty$}). 
\end{equation*}
\end{enumerate}
Moreover, if  $d_\alpha\ne 0$, $(X_n)_n$ has an asymptotic direction given by $d_\alpha$: 
\begin{equation*}
	P^{(\alpha)}_0\text{-a.s.,}\qquad \lim_{n\to\infty} \frac{X_n}{|X_n|}= \frac{d_\alpha}{|d_\alpha|}.\label{eq:direction}
\end{equation*}
\end{theorem}

\begin{remark}
Note that this result contains in particular the directional 0-1 law, which is still an open question for general RWRE in dimension $d\ge 3$
(the case $d=2$ was settled in \cite{zerner-merkl2001}).
\end{remark}

The proof of Theorem~\ref{th:transience_dir} contains two main parts. The first is a lower bound on the probability of directional transience (or actually on the probability of staying on one side of a hyperplane, see below), and the second is the 0-1 law mentioned just before. The 0-1 law is due to Bouchet~\cite{bouchet2013} (if $d\ge3$) and based on the results of Section~\ref{ss_ballistic}. Let us for the moment state the lower bound, which is actually an identity. 

Let $\ell$ be any direction in $\Q^d$ such that $\ell\cdot d_\alpha>0$. The discrete half-space 
\[\mathcal H^+_\ell=\{x\in\Z^d\,:\,x\cdot\ell\ge0\}\]
has a periodic boundary
\[\mathcal H^0_\ell=\{x\in\mathcal H^+_\ell\,:\,\exists i\in\{1,\ldots 2d\}\text{ such that } x+e_i\notin\mathcal H^+_\ell\}.\]
Let us introduce a probability distribution $\mu_\ell$ on some arbitrary period of $\mathcal H^0_\ell$ such that, for each point $x$, $\mu_\ell(x)$ is proportional to the sum of the weights of the edges entering~$\mathcal H^+_\ell$ at $x$. We may then consider the law $P^{(\alpha)}_{\mu_\ell}$ of the RWDE started at a point sampled according to $\mu_\ell$. 
This choice of $\mu_\ell$ introduces some stationarity and enables to obtain:

\begin{proposition}
	\label{prop:transience_dir_generale}
	Assume $\ell\in\Q^d$ and $\ell\cdot d_\alpha>0$. Then
	\begin{equation}
		P_{\mu_\ell}^{(\alpha)}\big(\forall n\ge0,\ X_n\cdot\ell\ge0\big)= 1-\frac{E^{(\alpha)}_0\big[(X_1\cdot\ell)_-\big]}{E^{(\alpha)}_0\big[(X_1\cdot\ell)_+\big]}>0.\label{eqn:inegalite_generale}
	\end{equation}
\end{proposition}

Note that the right hand side is fully explicit since the law of $X_1$ is given by the initial weights. It is surprising that the above probability does only depend on the initial weights up to a constant factor. In particular (thinking of~\eqref{eq:large-weights}), we can check that the above probability is the same as for a simple random walk in the mean environment, by applying the same proof (Lemma~\ref{lem:reversal} obviously holds for this walk). 

This results stated above appears in~\cite{tournier2015direction}. It was however first given as a lower bound and in the case $\ell=e_1$ in~\cite{sabot-tournier2011}. 
When $\ell=e_1$ the previous statement takes a simpler form :

\begin{corollary}
\label{cor:transience_dir}
Assume $\alpha_1>\alpha_{1+d}$. Then
\begin{align}
P^{(\alpha)}_0\big(\forall n\ge0,\ X_n\cdot \vec e_1\ge0\big)
	& = 1-\frac{\alpha_{1+d}}{\alpha_1}.
\label{eqn:inegalite}
\end{align}
\end{corollary}

A short proof of the lower bound of the corollary is given in Section~\ref{sub:proof_transience_dir}, where the reversal lemma plays a key role.

\subsection{Transience in dimension $d\ge 3$}

Transience is a translation invariant property of the environment, and as such, due to ergodicity, it holds for almost every or almost no environment. For Dirichlet environments, when $d_\alpha\ne 0$, transience follows from directional transience. The following theorem by Sabot~\cite{sabot2011transience}, which again makes crucial use of Lemma~\ref{lem:reversal}, covers in particular the remaining case $d_\alpha=0$, under the condition that $d\ge3$.  (In fact, ii) was not proved in \cite{sabot2011transience} but is based on the same principle, and proved in Section \ref{sec:proofs}.)

\begin{theorem}\label{th:transience}\ 
	\begin{enumerate}[(i)]
	\item Assume $d\ge3$. For any $\alpha_1,\ldots,\alpha_{2d}$, 
		\[P^{(\alpha)}_0\text{-a.s.,}\qquad \lim_{n\to\infty}|X_n|=+\infty.\] 
	More precisely, if $G_\omega(0,0)=\sum_{n=0}^\infty P_{0,\omega}(X_n=0)$ denotes the Green function in environment $\omega$, then
	\[\E^{(\alpha)}[G_\omega(0,0)^s]<\infty\qquad\Leftrightarrow\qquad s<\kappa.\]
\item Assume $d=2$. For any $\alpha_1,\alpha_2,\alpha_3,\alpha_4$, for any $\eps>0$, there is a constant $C$ such that for all $N\ge2$, if $H_N$ denotes the hitting time of $N$ by $(|X_k|)_k$, and $H_0^+$ the first return time to $0$, then
	\[P^{(\alpha)}_0(H_N< H_0^+)\ge\frac C{(\ln N)^{\frac12+\eps}}.\]
	\end{enumerate}
\end{theorem}

Remind that in dimension 2, for the (recurrent) simple symmetric random walk, the last probability is on the order of $(\ln N)^{-1}$, hence 2D annealed RWDE are ``strictly less recurrent'' than the symmetric simple random walk; this is in contrast with the identity of Proposition~\ref{prop:transience_dir_generale} which remains true for the simple random walk in the mean environment.
Let us stress that recurrence in dimension 2 with $d_\alpha=0$ remains open. 

A full proof of the first statement is provided in Section~\ref{sub:proof_transience}. It is worth noticing that it applies as well to transient Cayley graphs with translation invariant weights: if a group $G$ is generated by $S=\{e_1,\ldots,e_d\}$ and this set $S$ is stable by inversion, then we may consider its Cayley graph and endow it with weights $\alpha_{(g,ge_i)}=\alpha_i$ for any $g\in G$ and $1\le i\le d$, where $\alpha_1,\ldots\alpha_d$ are given parameters; under these assumptions, and transience of the simple random walk on $G$, this RWDE on $G$ is transient. \cite{sabot2011transience} gives a more general version, where less symmetry of the graph and weights is required.

Comparing with~\eqref{eq:finite_traps}, the second statement of the theorem suggests that there is no ``infinite trap'', in the sense that the integrability condition for $G_\omega(0,0)$ is the same for $\Z^d$ and for finite subsets of it. 
An analogy with the 1-dimensional case (see Section~\ref{sec:1D}) also suggests that $\kappa$ governs the order of fluctuations of $(X_n)_n$, a fact later confirmed in~\cite{sabot2013particle} and~\cite{bouchet2013}, cf.~next section.

\subsection{The invariant measure for the environment viewed from the particle}
\label{ss_ballistic}
For $x\in\Z^d$, denote by $\tau_x$ the shift on the environment,
defined by
$$
\tau_x\w (y,z)=\w(x+y,x+z), \;\;\;\forall \omega\in\Omega,\ \forall y,z\in\Z^d.
$$
Given the random walk $(X_n)_n$ in environment $\w$, the process of the environment
viewed from the particle is the process on the state space $\W$
defined by
$$
\overline \w_n=\tau_{X_n}\w, \;\;\;\forall n\in\N. 
$$
Under $P^{\w}_0$, $\w\in \W$, (resp.~under the annealed law $P_0$)
$\overline\w_n$ is a Markov process on state space $\W$ with
generator $R$ given by
$$
Rf(\w)=\sum_{i=1}^{2d} \w_{e_i} (f(\tau_{e_i} \w)-f(\w)),
$$
for all bounded measurable function $f$ on $\W$, and with initial
distribution $\delta_{\w}$ (resp.~$\P$), cf.~e.g.~\cite{sznitman-ten}. 
The advantage of this point of view is that the process $\overline\w_n$ contains more information and 
is a Markov process, even under the annealed law $P_0$.
One key ingredient to be able to apply this technique is the so-called
equivalence between static and dynamic point of view: it corresponds to the existence of an invariant measure
for the process of the environment viewed from the particle which is absolutely continuous with respect to the
initial law of the environment.

The point of view of the particle has been the central tool in the analysis of random walks in random conductances. 
It has had yet a little impact in the general non-reversible case, since it is still out of reach for general environments to prove
the equivalence of static and dynamic points of view.  It has been done only in few cases, the balanced case (\cite{Lawler-82,sznitman-ten}), 
the Dirichlet case in dimension $d\ge 3$, and under a ballisticity condition for general environments
in dimension $d\ge 4$, \cite{berger2014}. The Dirichlet case is the object of the  
the following theorem from \cite{sabot2013particle}.
\begin{theorem}\label{th:invariant}
Let $d\ge 3$ and $\P^{(\alpha)}$ be the law of the Dirichlet
environment with weights $(\alpha_1, \ldots, \alpha_{2d})$. Let
$\kappa>0$ be defined, as in~\eqref{eq:def_kappa}, by
$$
\kappa= 2\left(\sum_{i=1}^{2d} \alpha_i\right) -\max_{i=1, \ldots
, d} (\alpha_{i}+\alpha_{i+d}).
$$

(i) If $\kappa>1$ then there exists a unique probability
distribution $\Q^{(\alpha)}$ on $\W$ absolutely continuous with
respect to $\P^{(\alpha)}$ and invariant by the generator $R$.
Moreover ${\d\Q^{(\alpha)}\over \d\P^{(\alpha)}}$ is in $L^p(\P^{(\alpha)})$ for
all $1\le p<\kappa$.

(ii) If $\kappa\le 1$, there does not exist any probability
measure invariant by $R$ and absolutely continuous with respect to
the measure $\P^{(\alpha)}$.
\end{theorem}

In the case (ii), the non existence of an absolutely continuous invariant measure viewed from the particle is due
to the presence of traps of finite size (cf.~Section~\ref{traps}). Indeed, the expected exit time of the RWDE in finite
size traps described in Section~\ref{traps} is infinite for $\kappa\le 1$. This implies that the RWDE spends most of its time
on configurations where the trapping effect is strong, hence on atypical configurations.

Theorem \ref{th:invariant} has consequences on the equivalence between directional transience and ballisticity 
and on the question of directional transience (cf.~Section~\ref{sub:transience_dir}).
To get a full picture it is important to deal with the case $\kappa\le 1$.
In \cite{bouchet2013}, E. Bouchet extended the previous result to the case $\kappa\le 1$, by considering an accelerated process.
The process is accelerated through a function that takes into account the local configuration: the accelerated process
goes faster when the local environment is strongly trapping.
For this accelerated process, it is possible to prove the existence of an invariant measure viewed from the particle. 

The strategy is the following: we fix a finite connected subset $\Lambda\subset \Z^d$ containing 0,
then at each vertex $x$, we define an accelerating function that will ``kill'' all traps strictly contained in the box $x+\Lambda$.
More precisely, recalling the notation (\ref{wsigma}),
the accelerating function is:
\begin{equation}
\label{gamma_omega}
\gamma ^\omega (x) = \frac{1}{\sum \omega _\sigma} , 
\end{equation}
where the sum is on all $\sigma$ finite simple directed paths from $x$ to 
$x + \Lambda$, (i.e.\ each vertex is visited at most once, and the path is stopped just after exiting $x+\Lambda$). 
Let $Z_t$ be the continuous-time Markov chain whose jump rate from $x$ to $y$ is $ \gamma ^\omega (x)  \omega (x,y) $, with $Z_0 = 0$. 
Then $Z_t$ is an accelerated version of the discrete time process since it jumps faster when $X_n$ is at a point $x$
such that $\gamma_x$ is small, i.e.\ when it is difficult to exit the box $x+\Lambda$. The environment viewed from the position of
$Z_t$, $\tau_{Z_t}\w$, has generator
\[ R^\Lambda f (\omega) = \sum _{i=1} ^{2d} \gamma ^\omega (0) \omega (0, e_i) \left( f(\tau _{e_i} \omega ) - f(\omega) \right) ,\]

\begin{theorem}\label{Elodie}

Let $d \geq 3$ and $ \P ^{(\alpha)} $ be the law of the Dirichlet environment for the weights $ (\alpha _1, \dots , \alpha_{2d}) $. Let $\kappa ^\Lambda > 0$ be defined by
\[ \kappa ^\Lambda = \min \Big\{ \sum _{e \in \partial_+(K)}  \alpha _e \,:\, K \text{ connected set of vertices},\, 0 \in K \text{ and } \partial \Lambda \cap K \neq \emptyset \Big\},
\footnote{cf.\ the illustration on picture \ref{pic:lambda}} \]
where  $\partial_+(K) = \{ e \in E\,:\, \underline{e} \in K , \; \overline{e} \notin K \}$ and $ \partial \Lambda = \{ x \in \Lambda\,:\, \exists y \sim x \text{ such that } y \notin \Lambda \} $.
If $\kappa ^\Lambda > 1$, there exists a unique probability measure $ \Q ^{(\alpha)} $ on $ \Omega $ that is absolutely continuous with respect to $ \P ^{(\alpha)} $ and invariant for the generator $R^\Lambda$. Furthermore, $ \frac{ \d \Q^{(\alpha)} }{\d \P^{(\alpha)}} $ is in $ L^p (\P ^{(\alpha)}) $ for all $ 1 \leq p < \kappa ^\Lambda $.

\end{theorem}

\begin{remark}
The parameter $\kappa_\Lambda$ measures the strength of the strongest finite trap containing 0 and not contained in the set
$\Lambda$ (compare with Section~\ref{traps}). Indeed, the condition $\partial \Lambda \cap K \neq \emptyset$ ensures that there is an edge from
$K$ to $\Lambda^c$, hence that there is a path inside $K$ that goes from 0 to $\Lambda^c$. 
\end{remark}
\begin{remark}
If $\Lambda$ is a box of radius $r _\Lambda$, the formula is explicit:
\[ \kappa ^\Lambda =  \min _{i_0 \in [\![ 1,d ]\!]} \bigg(  \alpha _{i_0} + \alpha _{i_0 + d} + (r_\Lambda + 1) \sum _{i \neq i_0} (\alpha _{i} + \alpha _{i + d})  \bigg) .\]
In particular, $\kappa_\Lambda$ can be made arbitrarily large, hence the equivalence between static and
dynamic point of view has a positive answer for an accelerated process for arbitrary weights.  
\end{remark}

\begin{center} 
\begin{figure}[h]
\begin{tikzpicture}[scale=0.8,>=stealth]
\draw [ultra thin, gray] (-2,-4) grid (6,3);
\tikzstyle{fleche}=[thick, dashed, ->] 
\draw[fleche] (0,0) -- (-1,0) ;
\draw[fleche] (0,0) -- (0,1) ;
\draw[fleche] (0,0) -- (0,-1) ;
\draw[fleche] (2,0) -- (3,0) ;
\draw[fleche] (1,0) -- (1,1) ;
\draw[fleche] (1,-1) -- (0,-1) ;
\draw[fleche] (1,-1) -- (1,-2) ;
\draw[fleche] (2,0) -- (2,1) ;
\draw[fleche] (2,-1) -- (3,-1) ;
\draw[fleche] (2,-2) -- (1,-2) ;
\draw[fleche] (2,-2) -- (2,-3) ;
\draw[fleche] (3,-2) -- (3,-1) ;
\draw[fleche] (3,-2) -- (3,-3) ;
\draw[fleche] (4,-2) -- (4,-1) ;
\draw[fleche] (4,-2) -- (4,-3) ;
\draw[fleche] (4,-2) -- (5,-2) ;
\draw[ -, very thick] (0,0) -- (1,0) -- (1,-1) -- (2,-1) -- (2,-2) -- (3, -2) -- (4,-2);
\draw[ -, very thick] (1,0) -- (2,0) -- (2,-1);
\draw [thick, dotted] (-1,2) -- (3,2) -- (3,1) -- (4,0) -- (4,-3) -- (-1,-3) -- (-1,2) ;
\draw (0,0) node[above right]{$0$} ;
\draw (3,2) node[above right]{$\Lambda$};
\end{tikzpicture}
\caption{$ \Sum _{e \in \partial_+(K)}  \alpha _e $ (dashed arrows) for an arbitrary $K$ (thick lines)\protect\footnotemark.\label{pic:lambda}}
\end{figure}
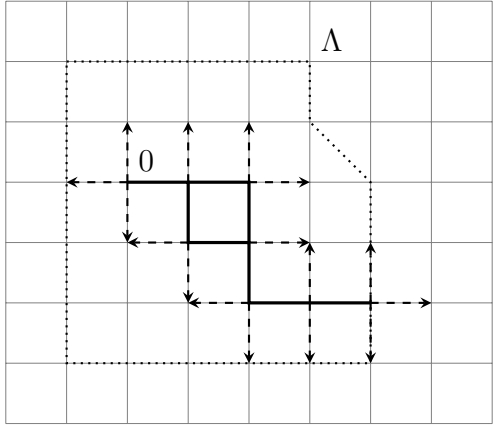
\end{center}
\footnotetext{The picture is borrowed from \cite{bouchet2013}, with kind authorization of E. Bouchet.}

\subsection{Ballistic regimes in dimension $d\ge 3$}

In dimension $d\ge 3$ it is possible to describe the family of parameters for which there is ballisticity.
It is contained in the next theorem and it is mainly a consequence of the existence of invariant measures for the process of the environment 
viewed from the particle.
\begin{theorem}
\label{ballisticity}
(i) Assume $d\ge 3$ and $d_\alpha\neq 0$. Then if $\kappa >1$, there exists $v\in \R^d\setminus\{0\}$ such that
$$
\lim_{n\to \infty} {X_n \over n} = v.
$$

(ii) If $d_\alpha=0$ or $\kappa\le 1$, then
$$
\lim_{n\to \infty} {X_n \over n} = 0.
$$

(iii) More precisely, if $d\ge 3$, $d_\alpha \cdot \ell > 0$ for some $\ell\in \R^d\setminus\{0\}$, and $\kappa\le 1$,  then

\[  \lim _{n \to + \infty} \frac{ \log( X_n \cdot \ell ) }{ \log (n) } = \kappa \text{ in } P_0^{(\alpha)} \text{-probability}.\]

\end{theorem} 
\begin{remark}
Note that from Theorem \ref{th:transience_dir},  $v$ is necessarily of the form $c\,d_\alpha$ for $c\in \R_+^*$.
\end{remark}
\begin{remark}
	The case (i) is proved in \cite{sabot2013particle} Theorem 2, it is a consequence of the existence of an absolutely continuous
invariant measure, cf.\ Theorem \ref{th:invariant} (cf.\ Section~\ref{ss_invariant_proof} for a sketch of proof).
The result (iii) comes from \cite{bouchet2013}, Theorem 1. It makes use of the accelerated process $Z_t$ defined in Section~\ref{ss_ballistic}.
For this accelerated process it is possible to prove a law of large numbers with non zero speed
when $\kappa_\Lambda>1$, $d\ge 3$ and $d_\alpha\neq 0$. From this process, we can recover the order of growth of the process
$X_n$.
The case (ii) when $d_\alpha=0$ is a simple consequence of the general law of large numbers of Zerner, \cite{zerner2002}.
When $\kappa\le 1$, it is a simple consequence of the existence of finite traps, cf.~Section~\ref{traps}, Proposition \ref{pro:traps}.
\end{remark}
\begin{remark}
A famous conjecture is that for uniformly elliptic RWRE in dimension $d\ge 2$, directional transience implies
ballisticity. We see that it is not true for non uniformly elliptic environment, in particular for Dirichlet environment.
Nevertheless, the previous theorem and the results of \cite{bouchet2013} tells that in dimension $d\ge 3$,  
it only comes from finite trap effects.
Indeed, Theorem 2.4 of \cite{bouchet2013} tells that when $\kappa_\Lambda>1$, 
$d_\alpha\neq 0$ implies that the accelerated process $Z_t$ has an asymptotic positive speed.
Hence, for a sufficiently large $\Lambda$, after acceleration by $\gamma^\w(x)$, which is a local function of the environment, 
directional transience implies ballisticity. It hence gives in dimension
$d\ge 3$ for Dirichlet environments an answer to the conjecture.
\end{remark}

It is expected that the same is true 
in dimension $d=2$, but we are still far for a proof of that. The argument used in the proof of the existence of the invariant measure viewed from the
particle is perfectly well suited for transient graphs, through the use of the existence of unit flows between two points with bounded $L^2$ norms, cf.~Section~\ref{ss_ballistic}. 
It is not clear whether an absolutely continuous invariant measure exists in dimension $d=2$, it is even believed that it does not exist in the case where
$d_\alpha=0$.

\section{Integrability of renewal times and functional central limit theorem in the directionally transient regime}
\label{sec:CLT}

Assume that $d_\alpha\cdot \ell \neq 0$ for some direction $\ell\in \R^d\setminus\{0\}$. 
We know from Theorem \ref{th:transience_dir} that the RWDE is transient in the
direction $\ell$. Denote by 
$(\tau_k)_{k\ge 0}$ the renewal times in the direction $\ell$, i.e.\ $\tau_0=0$ and, for all $k\ge0$, 
\[\tau_{k+1}=\inf\Big\{n> \tau_k\,:\,\forall i<n\le j,\ X_i\cdot\ell< X_n\cdot\ell\le X_j\cdot\ell\Big\}.\]
It follows from the transience in direction $\ell$ that these times are finite a.s.\ (cf.~\cite{sznitman-zerner1999}). 
If $\kappa>1$, in dimension $d\ge 3$, we know that the RWDE is ballistic 
and hence that the renewal times are integrable (integrability of renewal times is easily seen to be equivalent 
to positive speed, cf.~\cite{sznitman-zerner1999}). But the existence of the absolutely continuous invariant measure fails
to give any information about the finiteness of higher moments of renewal times, and finiteness of second 
moments of renewal times is a necessary ingredient to prove a central limit theorem. 
The next two sections answer partially the questions of integrability of renewal times and annealed/quenched central limit theorem.

\subsection{$(T)$ condition and integrability of renewal times}
\begin{definition}[\cite{sznitman2001}]
Let $\gamma\in (0,1)$. The $(T)_\gamma$ condition in the direction $\ell$ is satisfied if
$(X_n)$ is transient in the direction $\ell$ and if
there exists $c>0$ such that
$$
E_0\left[ \exp \left( c\sup_{0\le n\le \tau_1} | X_n|^\gamma \right)\right] <\infty,
$$
where $\tau_1$ is the first renewal time in the direction $\ell$.
We say that $(T)'$ is satisfied if $(T)_\gamma$ is satisfied for all $\gamma<1$.
\end{definition}
\begin{remark}\label{rem:polyn}
Equivalence between the $(T)_\gamma$ conditions 
and weaker polynomial conditions have been proved by Berger, Drewitz, Ram\'irez (\cite{berger-ramirez2014}), in the uniformly elliptic case,  and extended to
the weakly elliptic case (including Dirichlet cases) by Campos, Ram\'irez~(\cite{campos-ramirez2014}).
\end{remark}
In the case of uniformly elliptic environment in dimension $d\ge 2$, the $(T)'$ condition was proved to imply integrability
of all moments of renewal times. (Note that this is not true in dimension 1 due to the presence of traps.)
This is no longer true in the weakly elliptic case, in particular conterexamples are given by the Dirichlet case.
Indeed, when $\kappa<1$ and $d_\alpha\neq 0$, the RWDE is directionally transient and satisfies a law of large numbers with speed 0, hence
renewal times are not integrable (cf.~\cite{sabot-tournier2011} and Proposition~\ref{pro:traps}); still, one can find choices of the parameters so that $(T)'$ also holds (cf.~Theorem~5 in~\cite{bouchet-ramirez-sabot}). 
The question of integrability of renewal times in the weakly elliptic case was considered first by Campos, Ram\'irez (\cite{campos-ramirez2014}), then improved
by Bouchet, Ram\'irez, Sabot (\cite{bouchet-ramirez-sabot}), and Fribergh, Kious (\cite{kious-fribergh}).
We state below the specification to Dirichlet case of Theorem 1 of \cite{bouchet-ramirez-sabot}.
\begin{theorem}
\label{integrability}
Consider the RWDE with parameters $(\alpha_i)_{i=1, \ldots, 2d}$ in dimension  $d\ge 2$.
Assume that there exist $\ell$ such that $d_\alpha \cdot \ell >0$. Consider the renewal times $(\tau_k)_{k\ge 1}$ 
in the direction $\ell$. Assume that $(T)_\gamma$ is satisfied for some $\gamma \in (0,1)$.
Then $E^{(\alpha)}_0[(\tau_1)^s]<\infty$ if and only if $s<\kappa$.
\end{theorem}
\begin{remark}
	In fact $(T)_{\gamma}$ condition can be replaced by the polynomial condition $(P)_M$ mentioned in Remark~\ref{rem:polyn} for $M> 15d+5$, but
in the Dirichlet case we are actually able to prove the condition $(T)_1$ under a condition on the weights,
cf.\ forthcoming Theorem~\ref{Kalikow}.
\end{remark}
Sufficient conditions implying $(T)_1$ were given by Enriquez, Sabot (\cite{enriquez-sabot2006}) and improved
by Tournier (\cite{tournier2009integrability}). In fact Kalikow's condition was proved (we do not define Kalikow's condition here)
which is known to imply $(T)_1$ (cf.~\cite{sznitman2001}).
\begin{theorem}
\label{Kalikow}
Consider a RWDE $(X_n)_{n\ge 0}$ with parameters $(\alpha_i)_{i=1,\ldots, 2d}$.
If 
\begin{equation}
\label{kalikow-cond}
\sum_{i=1}^d \vert \alpha_i-\alpha_{i+d}\vert >1,
\end{equation}
then there exists a direction $\ell\in \R^d\setminus\{0\}$,  such that the
RWDE satisfies Kalikow's condition (cf.~\cite{kalikow1981}), hence $(T)_1$ condition in the direction $\ell$.
\end{theorem}
This was proved using an integration by part formula, very specific to the Dirichlet case.
The following is an easy corollary of Theorems \ref{integrability} and \ref{Kalikow}.
\begin{corollary}
\label{integrability-Dirichlet}
Assume that $d \ge 2$ and that $d_\alpha \cdot \ell >0$ for some direction $\ell$.
Assume that the condition~\eqref{kalikow-cond} is satisfied, and denote $(\tau_k)_{k\ge 1}$
the renewal times in the direction $\ell$.
Then $E_0\left[ \tau_1^s\right]<\infty$ if and only if $s<\kappa$.
\end{corollary}
It is expected that this condition is not optimal, and we conjecture that at least $(T)'$ is satisfied
when $d_\alpha\neq 0$, hence that renewal times have finite $s$-moments for $s<\kappa$ when the
RWDE is directionally transient, i.e.\ when $d_\alpha\neq 0$.
\subsection{Quenched functional central limit theorem}\label{def-FCLT}
Let $(X_n)_{n\ge 0}$ be a RWRE in $\Z^d$. 
Define a sequence of processes $ (B^{(n)}_t) _{t\geq0} $ by setting
\begin{equation}
B^{(n)}_t = \frac{X_{[nt]}-[nt]v}{\sqrt{n}}, \;\; t \ge 0,
\end{equation}
where $ [x] := \max \{ n \in \Z : n \leq x \}$ stands for the integer part of $x$. 
\begin{definition}
The RWRE $(X_n)$ satisfies an annealed functional central limit theorem (FCLT) with non degenerate covariance matrix, 
if $B^{(n)}$ converges weakly (for the Skorokhod topology) to a Brownian motion with non degenerate covariance matrix, under the annealed law
$P_0$.

The RWRE $(X_n)$ satisfies a quenched functional central limit theorem (FCLT) with non degenerate covariance matrix, 
if for $\P$-a.e.\ environment $\w$,  $B^{(n)}$ converges weakly (for the Skorokhod topology) to a Brownian motion with non degenerate covariance matrix, under the quenched law
$P_{0,\w}$.
\end{definition}

An annealed functional CLT has been proved by Sznitman, under a second moment condition
$E_0[(\tau_1)^2]<\infty$ for the renewal time, cf.~\cite{sznitman-00}.
Quenched functional CLT have been proved by Rassoul-Agha and Sepp\"al\"ainen in the weakly elliptic case
under very high moment conditions on renewal times (cf.~\cite{RS-Sepp}) and by Berger, Zeitouni, for uniformly elliptic case under
high moment condition (\cite{berger-zeitouni}). These proofs have been improved to get a near optimal moment condition on renewal times,
if a $(T)_\gamma$ condition is satisfied. This is the content of the next result from \cite{bouchet-sabot-dossantos}, that we state for simplicity
in the case of i.i.d.\ nearest neighbor RWRE.

\begin{theorem}
\label{theo:quenchedFCLT}
Set $d \geq 2 $. We consider a nearest neighbor random walk in an i.i.d.\ random environment. 
Assume that $(X_n)$ is transient in the direction $\ell \in \R ^d \setminus \{ 0 \}$, 
and denote by $\tau_1$ the corresponding regeneration time.
Suppose that the walk satisfies condition $ (T)_\gamma $ for some $ 0 < \gamma \leq 1 $ and
that $ E_0 \left[ \tau_1^{2} (\ln \tau_1 ) ^m \right] < + \infty $ for some $ m > 1 + \frac{1}{\gamma}$. 
Then, for $\P$-a.e.\ environment $\omega$, 
the process $B^{(n)}$ satisfies a quenched functional central limit theorem with a deterministic, 
non degenerate covariance matrix.
\end{theorem}
We deduce from Corollary~\ref{integrability-Dirichlet} and the previous theorem, the following result in the Dirichlet case.
\begin{theorem}
\label{FCLT-Dirichlet}
Assume that $d\ge 2$ and consider the RWDE $(X_n)_{n\ge 0}$ with weights $(\alpha_i)$.
Assume that the condition~\eqref{kalikow-cond} is satisfied, and that $\kappa>2$. Then $(X_n)_{n\ge 0}$ satisfies
a quenched functional central limit theorem with a deterministic, non degenerate, covariance matrix.
\end{theorem}

\section{Some sketches of proofs}
\label{sec:proofs}

In this section we sketch some proofs of Section~\ref{ss:results-reversal}, that use the time reversal property.
For Theorem~\ref{th:transience} and Theorem~\ref{th:invariant}, we do not optimize on the parameters which considerably simplifies the argument.
In Section~\ref{sub:optimization} we shortly explain how the optimization on the parameters is related to the max-flow min-cut theorem.

\subsection{Directional transience.}
\label{sub:proof_transience_dir}
Let us use the time reversal lemma~\ref{lem:reversal} in order to prove Corollary~\ref{cor:transience_dir}. 

We do the proof in $\Z^2$ to lighten notation; the general case is a straightforward generalization. 

\begin{figure}
\begin{center}
\includegraphics[width=9cm]{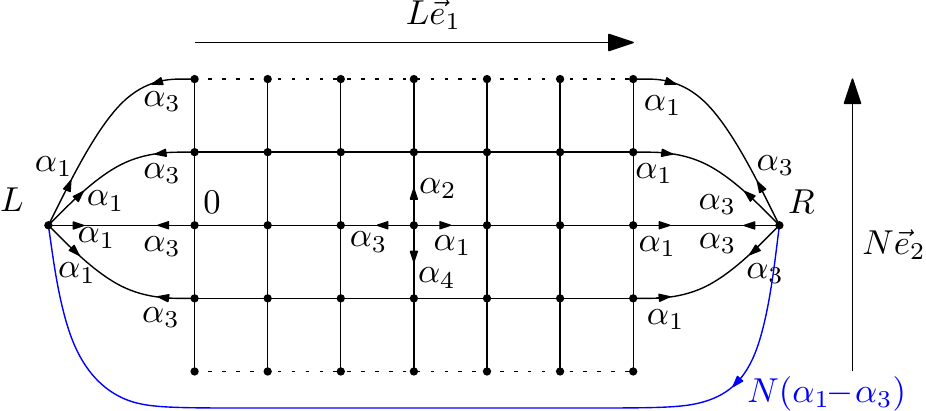}
\caption{Graph $G_{N,L}$ for the proof of directional transience in dimension~2 (Corollary~\ref{cor:transience_dir}). Vertices on top are identified with the corresponding vertices on bottom.}
\label{fig:transience_dir}
\end{center}
\end{figure}

Let $N,L\in\N$. We consider the finite graph $G_{N,L}$ defined by figure~\ref{fig:transience_dir}, endowed with the weights indicated on the figure. This is a horizontal cylinder (top and bottom lines of the figure are identified) of length~$L$ and circumference~$N$, together with two vertices $L$ and $R$ respectively corresponding to the left and right exits of the cylinder and a ``long'' edge $(R,L)$. The vertex $0$ belongs to the leftmost column of the cylinder. We first have, denoting respectively by $H$ and $H^+$ hitting and return times, 
\[P^{(\alpha)}_0(H_R<H_L)=P^{(\alpha)}_L(H_R<H^+_L),\]
due to vertical translation invariance of the graph and because, on the right hand side event, no return to $L$ occurs (so that the reinforcement of the first edge crossed doesn't affect the probability of the event). Then, using the existence of the long edge $(R,L)$, 
\begin{align*}
P^{(\alpha)}_L(H_R<H^+_L)
	& \ge P^{(\alpha)}_L(X_{H^+_L-1}=R).
\end{align*}
However, each sample of the right hand side event corresponds to stepping around a \emph{cycle} (viz., coming from $L$ and getting back to $L$ through the long edge). 
By applying Lemma~\ref{lem:reversal} (we have $\diverg\alpha=0$ in $G_N$) 
to each of these cycles, 
we get (see also~\eqref{returnproba})
\begin{align*}
P^{(\alpha)}_L(X_{H^+_L-1}=R)
	& = P^{(\alphach)}_L(X_1=R)
	 =\frac{N(\alpha_1-\alpha_{1+d})}{N\alpha_{1+d}+N(\alpha_1-\alpha_{1+d})}=1-\frac{\alpha_{1+d}}{\alpha_1}.
\end{align*}
Indeed, the reversal of the previous cycles are cycles starting by the long edge and coming back to $L$, but this second condition is almost surely satisfied since the graph is finite; and the explicit expression of the probability follows from the fact that the law of $X_1$ is given by the initial weights. In the end, we obtained 
\[P^{(\alpha)}_0(H_R<H_L)\ge 1-\frac{\alpha_{1+d}}{\alpha_1}.\]
By routine arguments, we may then take the limit as $N\to\infty$ and then $L\to\infty$ in order to get~\eqref{eqn:inegalite}. 

\begin{remark}
As shown in Section~\ref{sub:transience_dir}, there is actually equality in Corollary~\ref{cor:transience_dir}. The idea of the proof is used below in dimension 1 in the second proof of Proposition~\ref{lem:CL}; however, one needs the fact that directional transience holds almost-surely, which is nontrivial in higher dimensions (see~Section~\ref{sub:transience_dir}). As for the proof of transience in other directions, the main ingredients are essentially the same as above. 
\end{remark}

\begin{remark}
Once the almost sure transience is proved in sufficiently many directions, the existence of an asymptotic direction and its value follow from general arguments due to Simenhaus~\cite{simenhaus2007} on RWRE. Simenhaus indeed proves that if directional transience holds almost surely for a family of directions that span $\R^d$, then the walk has an asymptotic direction which is a constant. Assuming that transience holds for any $\ell\in\Q^d$ with $\ell\cdot d_\alpha>0$ (cf.~Proposition~\ref{prop:transience_dir_generale}), the direction then exists and may only be given by $d_\alpha$. 
\end{remark}

\subsection{Transience}
\label{sub:proof_transience}

Let us prove the first part of Theorem~\ref{th:transience}, i.e.~transience of the RWDE in $d\ge3$. The estimate (ii) for $d=2$ follows the same
ideas and is explained after.

Let $N\in\N$. We consider the graph $G_N$ with vertices $V_N=B(0,N)\cup\{\partial\}$ (where $\partial$ is an additional vertex and $B$ denotes a ball in $\Z^d$) and edges $\hat E_N=E_N\cup\{(\partial, 0)\}$ as in figure~\ref{fig:transience}, i.e.~of the following types.
The edges set $E_N$ is the set of
directed edges between neighboring vertices in $B(0,N)$ (as in $\Z^d$) and between $\partial$ and the vertices at the boundary of $B(0,N)$.
We also add to $E_N$ one special edge $(\partial,0)$. 

The weight $\alpha$ on $\Z^d$ naturally yields weights on $E_N$.
Let us endow the special edge $(\partial, 0)$ with a weight equal to 1.
With this choice we have clearly that on $\hat E_N$ 
$$
\diverg (\alpha)=\delta_\partial-\delta_0.
$$
(The divergence operator is defined in (\ref{divergence}).)
Consider now a unit flow $\theta:E_N\to \R_+$ from 0 to $\partial$ (i.e.\ $\dive(\theta)=\delta_0-\delta_\partial$) and assume that $0\le \theta\le 1$. 
(It is always possibly to find such a flow, cf.~below.) 
Extend $\theta$ by 0 on the special edge $(\partial, 0)$.
We consider the weights $\alpha'=\alpha+\theta$ which is clearly a flow with null divergence on $\hat E_N$.
\begin{figure}[h]
\begin{center}
\includegraphics[height=6cm]{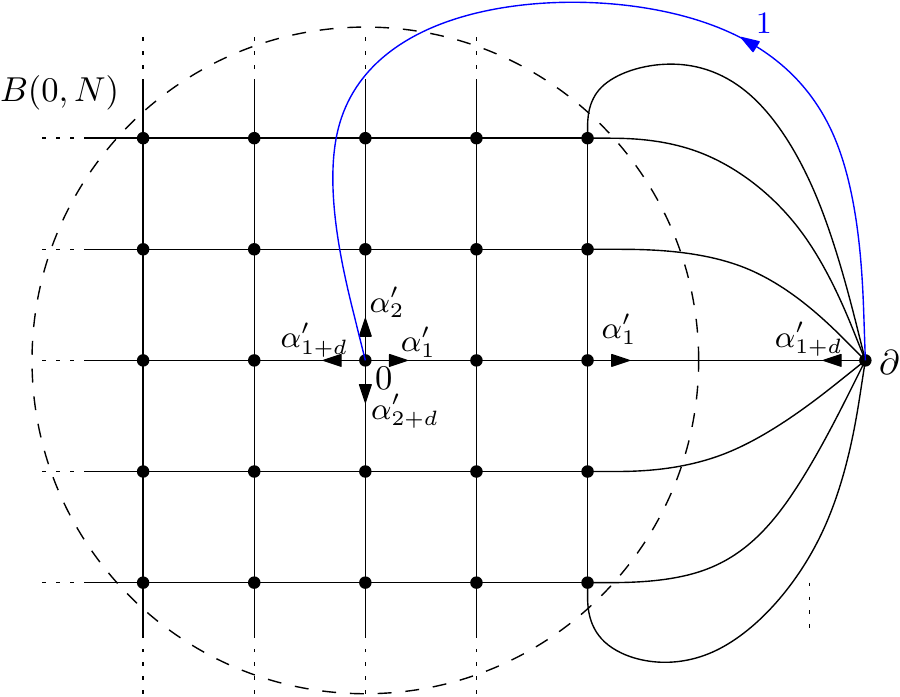}
\end{center}
\caption{Graph $G_N$ for the proof of transience.}
\label{fig:transience}
\end{figure}

\noindent Set as usual $H_\partial=\inf\{n\ge 0, \;\; X_n=\partial\}$ and $H_0^+=\inf\{n>0, \;\; X_n=0\}$.
We can now apply Lemma~\ref{lem:reversal} and its consequence (\ref{returnproba}): 
under the law $\P^{(\alpha+\theta)}$
\begin{align}\label{rrr}
P_{0,\w}
(H_\partial<H_0^+)
	& \ge P_{0,\w}
	(X_{H_0^+-1}=\partial)\\
	\nonumber
	& = P_{0,\check\w}
	(X_1=\partial)
	\\
	\nonumber
	&\sim \hbox{Beta}(1, \alpha_0).
\end{align}
In particular, we have
\begin{align}\label{rr}
P_{0}^{(\alpha+\theta)}
(H_\partial<H_0^+)
\ge
\frac1{1+\alpha_0}.
\end{align}
Provided $\theta$ comes from the restriction of a unit flow $\theta$ defined on the whole set $E$ of edges of $\Z^d$, letting $N\to\infty$ would give a positive probability of transience for the weights $\alpha+\theta$ without any condition on the dimension.

In order to get back to the weights $\alpha$, we use absolute continuity of Dirichlet measures (viz., $\frac{\d\P^{(\alpha)}}{\d\P^{(\alpha+\theta)}}=\frac{Z_{\alpha+\theta}}{Z_\alpha}\prod_e\omega_e^{-\theta_e}$, cf.~\eqref{eq:defZ}), and Hölder inequality for an arbitrary exponent $r>1$ (in the form $\E[UV^{-1}]\ge \E\big[U^{\frac1r}\big]^r\E\big[V^{\frac1{r-1}}\big]^{-(r-1)}$):
\begin{align*}P_0^{(\alpha)}(H_\partial<H_0^+)
	& =\mathbb E^{(\alpha+\theta)}[P_{0,\omega}(H_\partial<H^+_0)\prod_{e\in E_N} \omega_e^{-\theta_e}]\cdot \frac{Z_{\alpha+\theta}}{Z_\alpha}\\
	& \ge \E^{(\alpha+\theta)}[P_{0,\omega}(H_\partial<H^+_0)^{\frac1r}]^r\,\cdot\,\E^{(\alpha+\theta)}\Big[\prod_{e\in E_N} \omega_e^{\frac1{r-1}\theta_e}\Big]^{-(r-1)}\ \cdot\frac{Z_{\alpha+\theta}}{Z_\alpha}.
\end{align*}
The first expectation is greater that $P_0^{(\alpha+\theta)}(H_\partial<H_0^+)$ and thus than $\frac1{1+\alpha_0}$. The other factors are explicit and equal to
\[\Phi(\alpha,\theta,p)=\left(\frac{Z_{\alpha+\theta}}{Z_{\alpha+\frac r{r-1}\theta}}\right)^{r-1} \frac{Z_{\alpha+\theta}}{Z_\alpha}= \exp\bigg(\sum_{e\in E_N} \nu_r(\alpha_e,\theta_e)-\sum_{x\in V_N}\nu_r(\alpha_x,\theta_x)\bigg),\]
where
\[\nu_r(a,t)=r\ln\Gamma(a+t)-\ln\Gamma(a)-(r-1)\ln\Gamma\big(a+\frac r{r-1}t\big).\]
$\nu_r$ is a smooth function on $\{(a,t)\in\R^2:a>0,\, t>-a\}$ and one can check that for all $a>0$, 
$\nu_r(a,0)=0$ and $\partial_t \nu_r (a,0)=0$, hence there are $\eps,C_r>0$ such that $|\nu_r(a,t)|\le C_r t^2$ when $-\eps\le t\le2d$ and $a$ stays within a given compact subset 
of $(0,+\infty)$ containing all values of $\alpha_e$ and $\alpha_x$. Since $\theta_e\le 1$ and $\theta_x\le 2d$ for all $e,x$, 

\begin{align*}
P_0^{(\alpha)}(H_\partial<H_0^+)
	& \ge \frac1{(1+\alpha_0)^r} \exp\Big(-C_r\sum_{e\in E_N} \theta_e^2-C_r\sum_{x\in V_N}\theta_x^2\Big)\\
	& \ge \frac1{(1+\alpha_0)^r} \exp\Big(-C_r(1+2d)\sum_{e\in E_N}\theta_e^2\Big).
\end{align*}
The following is a simple application of Thomson's principle \cite[Chapter 2]{lyons-peres},
\begin{lemma}\label{flow}
There exists a unit flow $\theta$ on $E_N$ from 0 to $\partial$, such that $0\le\theta\le1$ and
$$
\sum_{e\in E_N} \theta_e^2 ={R_N},
$$
where $R_N$ is the electrical resistance between 0 and $B(0,N)^c$ for the network $\Z^d$ with unit resistance on the bonds.
\end{lemma} 
\begin{proof}
Consider the non-directed graph $(V_N, \tilde E_N)$ naturally associated with $(V_N, E_N)$ (without the special edge).
Thomson's principle \cite[Chapter 2]{lyons-peres} tells that
\begin{eqnarray}\label{min2}
R_N=\inf_{\tilde\theta} \|\tilde\theta\|^2_{L^2(\tilde E_N)},
\end{eqnarray}
where the sum runs on unit flows from 0 to $\partial$ on the non-directed graph: 
a unit flow on the non-directed graph is a signed function $\tilde\theta:\tilde E_N\to \R$
such that $\dive(\tilde\theta)=\delta_x-\delta_y$ (an arbitrary orientation on the edges must be chosen, cf.~\cite[Chapter 2]{lyons-peres}). 
From Proposition 2.2 and Exercise 2.35 of \cite{lyons-peres}, the flow that minimizes the $L^2$ norm satisfies $\vert \tilde\theta(e)\vert\le 1$.
Hence we can define $\theta$ on directed edges $E_N$ from the minimizer of (\ref{min2}), by assigning for every non-directed edge $e$ the value
$\vert\tilde\theta(e)\vert$ to the edge oriented according to the sign of $\tilde\theta(e)$ and $0$ to the reversed edge.
\end{proof}
\begin{remark}
For a nice construction of $\theta$ in $\Z^3$ using P\'olya urns, the reader is referred to~\cite{levin-peres2010}.
\end{remark}
Hence we get the inequality
\begin{align*}
P_0^{(\alpha)}(H_\partial<H_0^+)\ge 
\frac1{(1+\alpha_0)^r} \exp\Big(-C_r(1+2d)R_N\Big).
\end{align*}
In dimension $d\ge 3$, we know that $\sup_N R_N=R(0,\infty)<+\infty$ where  $R(0,\infty)$ is the electrical resistance between $0$ and $\infty$ for unit
resistances on bonds.
Letting $N\to\infty$ then yields 
$P_0^{(\alpha)}(H_0^+=\infty)\ge c>0$, hence $P_0^{(\alpha)}(|X|\to\infty)\ge c>0$. Finally, since, by ergodicity, $\omega\mapsto P_{0,\omega}(|X|\to\infty)$ is constant $\P^{(\alpha)}$-a.s., and on the other hand it is always equal to either 0 or 1, we conclude that this probability equals 1, $\P^{(\alpha)}$-a.s., which concludes the proof for
$d\ge 3$.

The 2-dimensional statement of Theorem~\ref{th:transience} is a consequence of the same proof as above, with the only differences that the special edge is given a small weight $\gamma_N$ instead of 1 so that we consider the divergence-free weight $\alpha'=\alpha+\gamma_N\theta$ on $G_N$.
This easily implies, following the same computation as in (\ref{rr}), that
\begin{align*}
P_0^{(\alpha+\gamma_N\theta)}(H_\partial<H_0^+)=\frac{\gamma_N}{\gamma_N+\alpha_0}.
\end{align*}
It is well-known that in dimension $d=2$, $R_N\sim \log N$, 	
 hence the lower bound becomes
	\[P_0^{(\alpha)}(H_N<H_0^+)\ge\left(\frac{\gamma_N}{\gamma_N+\alpha_0}\right)^r \exp(-5 C_r (\gamma_N)^2\ln N).\] 
	Taking $r=1+2\eps$ and $\gamma_N=(\ln N)^{-1/2}$ gives the lower bound of the theorem. 	

\begin{remark}
If we assume only $\diverg_x\alpha\ge0$ for all vertices $x$ and $0<c\le \alpha_e\le C$ for all edges $e$, then the above proof would work as well, provided we introduce new edges in $G_N$, from $\partial$ to every $x$ in $V$, with respective weight $\diverg_x\alpha$ (so that the new weight has divergence $\delta_0-\delta_\partial$ before addition of $\theta$). Due to orientation of these new edges (\emph{from} $\partial$ to $x$), they do not play a role in the probabilities involved in the proof. In this case, the lower bound becomes $\frac{1+\diverg_0(\alpha)}{1+\alphach_0+\diverg_0(\alpha)}=1-\frac{\alphach_0}{1+\alpha_0}$. Note that a positive divergence implies that the sum of the weights of edges exiting a finite subset is greater than the sum of those entering, which leads to an intuitive tendency to push the walk out of finite subsets. 
\end{remark}

The bound on the moments of the Green function are slightly more tricky, but 
if we do not optimize on the parameters it is rather easy to obtain integrability
of the Green function for $s<\tilde \kappa$ where 
$$\tilde \kappa=\min(\alpha_i,\; i=1,\ldots, 2d).$$
Consider the Green function $G_\w^N(0,0)$ of the quenched Markov chain in environment $\w$, killed at the exit time
of $B(0,N)$. We clearly have $G_\w^N(0,0)=1/P_{0,\w}(H_\partial<H_0^+)$.
Consider now the graph $(V_N, \hat E_N)$ as before but endow the special edge $(\partial , 0)$ with a weight $\gamma$.
Consider the weights $\alpha'=\alpha +\gamma\theta$. Proceeding as in \eqref{rrr},
under $\P^{(\alpha')}$, $P_{0,\w}(H_\partial<H_0^+)$ is stochastically bounded from below by the law
${\rm Beta}(\gamma, \alpha_0)$. Apply H\"older inequality with $r,q$ such that ${1\over r}+{1\over q}=1$, (in a direct way)
\begin{eqnarray*}
\E^{(\alpha)}\left[ G^N_\w(0,0)^s\right]&=&
 {Z_{\alpha+p\theta}\over Z_{\alpha}}\E^{(\alpha+\gamma\theta)}\left[ \left({1\over P_{0,\w}(H_\partial<H_0^+)}\right)^s\w^{-\gamma\theta}\right]
\\
&\le &
{Z_{\alpha+p\theta}\over Z_{\alpha}}\E^{(\alpha+\gamma\theta)}\left[ \left({1\over P_{0,\w}(H_\partial<H_0^+)}\right)^{rs}\right]^{1/r}
\E^{(\alpha+\gamma\theta)}\left[ \w^{-\gamma q\theta}\right]^{1/q}
\\
&\le &
{\left(Z_{\alpha+p\theta}\right)^{1-{1\over q}} \left(Z_{\alpha+(1-q)\theta}\right)^{{1\over q}} \over Z_{\alpha}} 
\E\left[ Y^{-rs}\right]^{1/r},
\end{eqnarray*}
where $Y$ is a random variable with distribution ${\rm Beta}(\gamma, \alpha_0)$.
By (\ref{moments})  $Z_{\alpha+(1-q)\theta}$ is finite if and only if 
\begin{eqnarray}\label{ineq-flow}
\gamma(q-1)\theta(e)<\alpha_e, \;\;\; \forall e\in E_N.
\end{eqnarray}
Since $\theta\le 1$, we can take $\gamma(q-1)<\tilde\kappa$ which means, in terms of $r$,
$r>{\tilde\kappa+\gamma\over \tilde \kappa}$.
Now, the left hand side is integrable if $rs<\gamma$ since  $P_{0,\w}(H_\partial<H_0^+)$ has law
${\rm Beta}(\gamma, \alpha_0)$. This means that we can take 
\begin{align}\label{s}
s<{\gamma\tilde\kappa\over \tilde\kappa+\gamma}.
\end{align}
For this choice of $s$, we can make a similar second order estimate of the Gamma terms and prove
that
$$
\E^{(\alpha)}\left[ G^N_\w(0,0)^s\right]\le \exp(C \|\theta\|^2).
$$
Hence for $s$ satisfying (\ref{s}), we have that $\sup_N \E^{(\alpha)}\left[ G^N_\w(0,0)^s\right]<+\infty$.
Taking $\gamma$ very large we can take $s$ up to $\tilde \kappa$.
In Section~\ref{sub:optimization}, we shortly explain how to get bounds on the moments up to
$\kappa$ instead of $\tilde \kappa$.

\subsection{Invariant measure viewed from the particle}
\label{ss_invariant_proof}

In this section we give a sketch of proof of the existence of an invariant measure of the process of the environment viewed from
the particle, Theorem \ref{th:invariant}, in the case where the weights $(\alpha_i)$ are sufficiently large, 
in fact when $\alpha_i >1$ for all $i$.  
This simplifies the argument, since in this case it is not necessary to optimize on the parameters
using the max-flow min-cut theorem.
Like for transience, the proof uses the time reversal property and the existence of unit flows with bounded $L^2$ norms in dimension
$d\ge 3$, this is
where the restriction on the dimension enters.

Following \cite{Lawler-82}, we first restrict to a large torus.
When $N\in \N^*$, we denote by $T_N=(\Z/N\Z)^d$
the $d$-dimensional torus of size $N$. We denote by
$G_N=(T_N,E_N)$ the associated directed graph image of the graph
$G=(\Z^d, E)$ by projection on the torus. 
We denote by $\W_N$ the space of environments on the torus $G_N$ (following Section~\ref{sec:rwde}).
We denote by $\P_N^{(\alpha)}$ the Dirichlet law on the torus of size $N$, with weights
$(\alpha_{(x,x+e_i)}=\alpha_i)_{i=1,\cdots, 2d}$.

For $\w$ in
$\W_N$ we denote by $(\pi_N^\w(x))_{x\in T_N}$ the invariant
probability measure of the Markov chain on $T_N$ with transition probabilities $\w$ (it is unique since the environments are
elliptic). Let
$$
f_N(\w)= N^d \pi_N^\w(0),
$$
and
$$
\Q_N^{(\alpha)}=f_N\cdot \P_N^{(\alpha)}.
$$
Thanks to translation invariance, $\Q_N^{(\alpha)}$ is a probability measure on
$\W_N$.
The following lemma, which implies Theorem \ref{th:invariant}, is proved in \cite{sabot2013particle}
\begin{lemma}\label{main-lemma}
Let $d\ge 3$. For all $p\in [1,\kappa[$
$$
\sup_{N\in \N} \| f_N\|_{L^p(\P^{(\alpha)}_N)} <\infty.
$$
\end{lemma}
It is standard that such an estimate implies Theorem \ref{th:invariant}, it is done for example
in \cite{sznitman-ten} in a very similar context, and precisely in \cite{sabot2013particle}.
We will sketch the proof of this lemma only for $p\in [1, \tilde \kappa)$, where 
$$
\tilde \kappa =\inf\{ \alpha_i,\; i=1,\ldots,2d\}.
$$
 It implies the existence of an absolutely continuous invariant measure when $\tilde \kappa>1$,
  which is in $L^p(\P^{(\alpha)})$ for
 $p< \tilde \kappa$. 
 \begin{proof} (Sketch of proof of Lemma~\ref{main-lemma}). 
 
\noindent {\bf Step1.} Let $(\w_{x,y})_{x\sim y}$ be in $\W_N$. Recall that the
time-reversed environment is defined by
$$
\check \omega_{x,y}=\pi_N^\w(y) \w_{y,x}{1\over \pi_N^\w(x)},
$$
for $x$, $y$ in $T_N$, $x\sim y$. At each point $x\in T_N$
$$
\sum_{\ue=x}\alpha(e)=\sum_{\oe=x}\alpha(e) =\sum_{j=1}^{2d} \alpha_j,
$$
It implies by  Lemma~\ref{lem:reversal}  that if $\w$ is distributed according to $\P_N^{(\alpha)}$, 
then $\check \w$ is distributed according to
$\P_N^{(\check \alpha)}$ where $\check \alpha$ are the parameters obtained by central symmetry from $\w$,
i.e.\ we have with $\check \alpha_{x,x-e_i}=\alpha_{x-e_i,x}=\alpha_i$. 

Let $p$ be a real, $ p\ge 1$, then
\begin{eqnarray}\label{inequality}
(f_N)^p=
(N^{d}\pi_N^\w(0))^p
= \left({\pi_N^\w(0)\over {1\over
N^d}\sum_{y\in T_N} \pi_N^\w(y)}\right)^p
\le  \prod_{y\in T_N} \left({\pi_N^\w(0)\over
\pi_N^\w(y)}\right)^{p/N^d},
\end{eqnarray}
where in the last inequality we used the arithmetico-geometric
inequality.
For $\theta:E_N\rightarrow \R_+$, we define $\check
\theta$ by
$$
\check \theta_{(x,y)}=\theta_{(y,x)}, \;\;\; \forall x\sim y.
$$
For any two functions $\gamma$ and $\beta$ on $E_N$, (resp. on $T_N$) we write 
$$\gamma^\beta:=\prod_{e\in E_N}
\gamma(e)^{\beta(e)} \;\;\; (\hbox{resp.} \;  \gamma^\beta:= \prod_{x\in T_N}
\gamma(x)^{\beta(x)}\; ).
$$
The following simple computation is important:
\begin{eqnarray}
\nonumber
{\check\w^{\check\theta}\over \w^{\theta}}&=& 
\prod_{e\in E_N} {\left( \w(e)\pi_N(\ue)\pi_N(\oe)^{-1}\right)^{\theta(e)}\over 
\w_e^{\theta(e)}}\,
\\
\nonumber 
&=&\prod_{x\in T_N} \pi_N(x)^{\sum_{e,\; \ue=x}\theta(e)-\sum_{e,\oe=x} \theta(e)} 
\\
\label{quotient} 
&=& \pi_N^{\diverg (\theta)}.
\end{eqnarray}
 Hence, for all $\theta:E_N\to \R_+$
such that
\begin{eqnarray}\label{dive-theta}
\dive (\theta)={1\over N^d}\sum_{y\in
T_N}(\delta_0-\delta_y),
\end{eqnarray}
we have using (\ref{inequality}) and (\ref{quotient})
\begin{eqnarray}\label{fN}
f_N^p&\le&
{\check\w^{p\check\theta}\over \w^{p\theta}}.
\end{eqnarray}

\noindent {\bf Step 2.}
We will use the following fact.
\begin{proposition}
There exists a constant $c(d)$ such that for any $N$, for any $x$ and $y$ in $T_N$, there exists
a function $\theta^{(x,y)}:\E_N\to \R_+$ such that
\begin{align*}
&\dive(\theta)= \delta_x-\delta_y,
\\
&0\le \theta^{(x,y)}(e)\le 1, \;\;\; \forall e\in E_N,\\
\text{and }&\sum_{e\in E_N} \left(\theta^{(x,y)}(e)\right)^2\le c(d).
\end{align*} 
\end{proposition}
\begin{proof}
Proceeding as in lemma \ref{flow} we can construct a unit flow from $x$ to $y$ such that
$0\le\theta\le 1$ and such that $\| \theta\|^2_{L^2(E_N)}=R_N(x,y)$ where $R_N(x,y)$ is the electrical resistance between $x$ and $y$ 
for the
torus network with unit conductances.  It is well-known that in dimension $d\ge 3$, there exists a constant $c(d)$ such that
$R_N(x,y)\le c(d)$ for all $N$ and all $x,y$ in $T_N$.
In \cite{sabot2013particle} a proof with an explicit construction is given since more precise information are necessary.
\end{proof}
From the previous proposition, we can construct
\begin{eqnarray}\label{theta}
\theta ={1\over N^d}\sum_{y\in T_N} \theta^{(0,y)}.
\end{eqnarray}
Clearly $\| \theta\|^2_{L^2(E_N)}\le c(d)$, $0\le \theta(e)\le 1$ for all edge $e\in E_N$, and 
\begin{eqnarray}\label{divetheta}
\dive (\theta)={1\over N^d}\sum_{y\in
T_N}(\delta_0-\delta_y).
\end{eqnarray}
Assume that $\tilde \kappa>1$, and $1<p<\tilde \kappa$.
Let $r,q$ be positive reals such that ${1\over r}+{1\over q}=1$
and $pq<\tilde\kappa$. Using (\ref{fN}), H\"older inequality and Lemma~\ref{lem:reversal} we get for $\theta$ defined
by (\ref{theta})
\begin{eqnarray}
\nonumber
&&\E^{(\alpha)} \left[f_N^p\right] 
\\
\nonumber
&\le & \E^{(\alpha)} \bigg[ {\check\w^{p\check\theta}\over \w^{p\theta}}\bigg]
\\
\nonumber
&\le & \E^{(\alpha)} \left[\check\w^{pr\check\theta}\right]^{1/r}  \E^{(\alpha)} \left[\w^{-pq\theta}\right]^{1/q}
\\
\nonumber
&=&
 \E^{(\check \alpha)} \left[\w^{pr\check\theta}\right]^{1/r}  \E^{(\alpha)} \left[\w^{-pq\theta}\right]^{1/q}
 \\
 \label{gam}
 &=&
 \left( {\prod_{e \in E_N} \Gamma(\check\alpha_e+pr\check\theta(e))\over \prod_{x\in T_N} 
 \Gamma(\alpha+pr\check\theta(x))}\right)^{1/r}
 \!\!\left( {\prod_{e \in E_N} \Gamma(\alpha_e-pq\theta(e))\over \prod_{x\in T_N} \Gamma(\alpha-pq\theta(x))}\right) ^{1/q}
 \!\!\left( \prod_{x\in T_N} \Gamma(\alpha) \over {\prod_{e \in E_N} \Gamma(\alpha_e)}\right),
\end{eqnarray}
where in the last expression we write $\alpha=\sum_{i=1}^{2d} \alpha_i$, 
and $\theta(x)=\sum_{i=1}^{2d} \theta_{(x,x+e_i)}$,  $\check\theta(x)=\sum_{i=1}^{2d} \check\theta_{(x,x+e_i)}$.
Remark that we need that
\begin{eqnarray}\label{cond-theta}
pq\theta(e)<\alpha_e, \;\;\;\forall e\in E_N,
\end{eqnarray} 
for the expectation in the middle term to be finite.
Since $pq\le \tilde\kappa$ and $0\le \theta(e)\le 1$, 
we have $\alpha_e-pq\theta(e)>(\tilde\kappa-pq)>0$ and $\alpha-pq\theta(x)>(\tilde\kappa-pq)>0$.
Hence the right hand side term is well-defined and finite.
Remark from (\ref{divetheta}), that 
$$
\check\theta(x) =\theta(x) -\indic_{\{x=0\}}+{1\over N^d}.
$$
Change now $e$ in $\check e$ in the product in (\ref{gam}), it gives
\begin{eqnarray*}
\nonumber
&&\E^{(\alpha)} \left[f_N^p\right] 
\\
&\le &
 {\prod_{e \in E_N} \Gamma(\alpha_e+pr\theta(e))^{1/r}\Gamma(\alpha_e-pq\theta(e))^{1/q}  \Gamma(\alpha_e)^{-1}
 \over \prod_{x\in T_N} 
 \Gamma(\alpha+pr(\theta(x)-\indic_{\{x=0\}}+{1\over N^d}))^{1/r}\Gamma(\alpha-pq\theta(x))^{1/q} \Gamma(\alpha)^{-1}}.
\end{eqnarray*}
Consider the numerator, it can be written
$$
\exp\left( \sum_{e\in E_N} \eta(\alpha_e, \theta_e)\right)
$$
with
\begin{eqnarray*}
\eta(a,t)= {1\over r} \ln \Gamma(a+pr t )+ {1\over q}\ln \Gamma(a-pq t)-\ln \Gamma(a).
\end{eqnarray*}
For small $t$, the 0th and 1st order vanish and we have
$$
\vert \eta(a,t) \vert \le O(t^2),
$$
for $a, t$ in a compact and $a- pqt>\epsilon$ for some $\epsilon >0$.
It implies that there is a constant $C>0$ such that
$$
\exp\left( \sum_{e\in E_N} \eta(\alpha_e, \theta_e)\right)\le \exp(C \| \theta\|^2_{L^2(E_N)})\le \exp (C c(d)).
$$
The argument for the denominator is similar, the extra term ${1\over N^d}$ just gives an extra constant and it can be 
bounded from below by $\exp(-C(1+ \| \theta\|^2))$ for some constant $C>0$.
\end{proof}

\subsection{Optimization on the parameters}
\label{sub:optimization}
In Sections~\ref{sub:proof_transience} and \ref{ss_invariant_proof}, we gave sketches of proof of the theorems under weaker conditions on the parameters.
We remark that the constraint on the parameters appears after the H\"older inequality, in equalities (\ref{ineq-flow}) for the estimates of the moments 
of the Green function and (\ref{cond-theta}) for the proof of the invariant measure. These inequalities are both of the
following type: we want to find $\gamma$ as large as possible such that 
$\gamma\theta(e)<\alpha_e$ for all edge $e$ for a unit flow from a point $x$ to another point $y$
(these are $0$ and $\partial$ in (\ref{ineq-flow}), and $0$ and a point $y$ in the torus in (\ref{cond-theta})).
Hence, we need to construct some flows that minimize the infimum
$\min(\theta(e)/\alpha_e)$.
This is exactly the content of the Max-Flow Min-Cut theorem that we recall below. 

Let $G=(V,E)$ be a directed connected graph. 
A flow from $x$ to $y$ is a function $\theta : E\to \R_+$ such that
$\dive(\theta)=\gamma(\delta_x-\delta_y)$ for some real $\gamma>0$.
The strength of $\theta$ is by definition $\gamma=\dive(\theta)(x)$.

Recall that a cutset separating $x$ from $y$ is a subset $S\subset E$ such that any directed
simple path from $x$ to $y$ contains at least one edge of $S$.
The well-known Max-Flow Min-Cut theorem says that the maximum flow
equals the minimal cutset sum (cf.~\cite{Ford}). We give here a
version for countable graphs (\cite{lyons-peres}, Theorem 2.19, cf.~also
\cite{Aharoni}).
\begin{proposition}\label{maxflow}
Let $(c(e))_{e\in E}$ be a
family of non-negative reals, called the capacities.
A flow $\theta$ from $x$ to $y$ is called
compatible with the capacities $(c(e))_{e\in E}$ if
$$
\theta(e)\le c(e), \;\;\; \forall e\in E.
$$
The maximum compatible flow equals the infimum of the cutset sum,
i.e.
\begin{eqnarray*}
\nonumber
&&\max\{\strength(\theta)\,:\, \hbox{$\theta$ is a flow from
$x$ to $y$ compatible with $(c(e))$}\}
\\
\label{Mincut}
&=&\inf\{c(S)\,:\,\hbox{$S$ is a cutset separating $x$ from
$y$}\},
\end{eqnarray*}
where
$$
c(S)=\sum_{e\in S} c(e).
$$
\end{proposition}

Hence, we see that the limitation $p \theta(e)\le \alpha_e$ is a max-flow problem.
This implies that for $c(e)=\alpha_e$, for all $p$ smaller than the min-cut, we can find a flow such that
$p\theta(e)< \alpha_e$. Consider for example the case of transience, Section \ref{sub:proof_transience_dir}.
The minimal cutset from 0 to $\partial$, on $\Z^d$ with capacities $\alpha_e$ is the set of edges
$\{(0,e_i)\,:\, i=1, \ldots, 2d\}$. Hence, the min-cut is $\sum_{i=1}^{2d} \alpha_i$.
This is strictly smaller than $\kappa$, and hence it does not match the optimal condition on $\alpha$.
This can be explained as follows: the minimal cutset corresponds to the ``trap'' with the single vertex
$\{0\}$. But this cannot be a trap since the RWDE makes a jump at each step.
This problem can be solved by the following trivial remark: we have 
$$
1=\sum_{i=1}^{2d} \w(0,e_i).
$$
This rather trivial identity indeed means that the RWDE leaves the vertex $\{0\}$ with probability 1 after one step!
The idea is to inject  $\sum_{i=1}^{2d} \w(0,e_i)$ in the identities. This has the effect of increasing the weight of at least one
edge in $\{e\,:\,\vert e\vert=1\}$ by 1. Hence, it makes it easier to create a flow compatible with these new weights.
This gives the optimal condition involving the parameter $\kappa$.
One technical difficulty that appears in \cite{sabot2011transience} and \cite{sabot2013particle} comes from the fact that we need to
implement the max-flow min-cut theorem, with keeping a finite $L^2$ norm. This is overcome by some surgery 
on the flows.

\section{The case of dimension 1. Relation with Chamayou Letac exact solutions of renewal equation}
\label{sec:1D}

\subsection{Generalities}

In the one-dimensional case, the environment $\omega$ is fully given by the sequence $(\omega_x)_{x\in\Z}$ of i.i.d.~real random variables $\omega_x=\omega_{(x,x+1)}$. In the Dirichlet case, their common distribution is $\Beta(\alpha_1,\alpha_2)$, henceforth denoted $\Beta(\alpha,\beta)$:
\[\text{under }\P^{(\alpha,\beta)},\qquad \omega_0\sim\Beta(\alpha,\beta)=\frac1{B(\alpha,\beta)}u^{\alpha-1}(1-u)^{\beta-1}\indic_{(0,1)}(u)\d u.\]
We may assume $\alpha\ge\beta$ without loss of generality, because of reflection symmetry. 

Many results are known in wide generality for RWRE on $\Z$ and can be readily applied to this situation. In particular, Solomon's results~\cite{solomon1975} give
\begin{theorem} $(d=1)$
\begin{description}
	\item[Transience] 
\begin{align*}
\text{if $\alpha=\beta$, }\qquad& -\infty=\liminf_n X_n<\limsup_n X_n=+\infty,\qquad P_0^{(\alpha,\beta)}\text{-a.s.}\\
\text{if $\alpha>\beta$, }\qquad& \qquad X_n\limites{}{n\to\infty}+\infty,\qquad P_0^{(\alpha,\beta)}\text{-a.s.},
\end{align*}
	\item[Ballisticity]
\begin{align*}
\text{if $\beta+1\ge \alpha\ge\beta$, }\qquad& \frac{X_n}n\limites{}{n\to\infty}0,\qquad P_0^{(\alpha,\beta)}\text{-a.s.},\\
\text{if $\alpha>\beta+1$, }\qquad& \frac{X_n}n\limites{}{n\to\infty}v=\frac{\alpha-\beta-1}{\alpha+\beta-1}>0,\qquad P_0^{(\alpha,\beta)}\text{-a.s.} 
\end{align*}
\end{description}
\end{theorem}

Annealed scaling limits for RWRE were also proved by Kesten, Kozlov and Spitzer~\cite{kks} and made more explicit in~\cite{limitlaws,stablefluctuations}. They are driven by the exponent ${\kappa_1}>0$ such that $\E[(\frac{1-\omega_0}{\omega_0})^{\kappa_1}]=1$. In the case of Beta environment, a simple computation gives
\[\fbox{${\kappa_1}=\alpha-\beta$}.\] 
This quantity plays a role analog to that of the constant ${\kappa}$ that we introduced in higher dimensions.
The fully explicit statement of the annealed scaling limits (Theorem~\ref{thm:kks} below) requires a computation that is specific to Beta environments and due to Chamayou and Letac~\cite{chamayou-letac}.  

\subsection{Chamayou and Letac's exact computation}

We assume $\alpha>\beta$ in the following. For $k\in\Z$, let
\[\rho_k=\frac{1-\omega_k}{\omega_k}\]
and define the random series
\[R=1+\rho_1+\rho_1\rho_2+\rho_1\rho_2\rho_3+\cdots.\]
The condition $\alpha>\beta$ implies $\E[\ln\rho_1]<0$ and thus a.s.~convergence of $R$. Then we have (Kesten~\cite{kesten1973})
\[P(R>t)\equivalent{t\to\infty} C_K(\alpha,\beta) t^{-{\kappa_1}},\]
where $C_K(\alpha,\beta)$ is known as Kesten's constant and is in general not explicit. Remarkably, it is however the case for Beta environment where even the law of $R$ is known, due to the work of Chamayou and Letac~\cite{chamayou-letac}:

\begin{proposition}\label{lem:CL}
	The random variable $1/R$ follows a distribution $\Beta(\alpha-\beta,\beta)$. In particular, $\P^{(\alpha,\beta)}(R\in \d r)=\frac1{B(\alpha-\beta,\beta)}\frac{(r-1)^{\beta-1}}{r^\alpha}\indic_{(0,+\infty)}(r)\d r$ and
\[C_K(\alpha,\beta)=\frac1{(\alpha-\beta)B(\alpha-\beta,\beta)}.\]
\end{proposition}

Let us emphasize that the random variable $R$ is involved in several quenched quantities, whose law is therefore explicit: for instance, simple computations give
\[P_{1,\omega}(H_0=+\infty)=\frac1R,\]
\[G_\omega(0,0)=\frac 1{\omega_0}R,\]
\[E_{0,\omega}[H_1]=2R_--1,\]
where $R_-=1+\rho_0+\rho_0\rho_{-1}+\cdots$ has same law as $R$. In particular, one can see that
\[\E^{(\alpha)}[G_\omega(0,0)^s]<\infty\qquad\Leftrightarrow\qquad s<{\kappa_1},\]
to be compared with Theorem~\ref{th:transience} for a higher dimensional analog. Furthermore, Kesten's constant appears in the scaling limits of~\cite{limitlaws,stablefluctuations} and thus in Theorem~\ref{thm:kks} below. 

\textit{Proof by renewal equation.}
The approach of Chamayou and Letac is based on the following equation: writing $\rho=\rho_1$, 
\begin{equation}
R=1+\rho R',
\end{equation}
where $R'$ has same distribution as $R$ and is independent of $\rho$. This distributional fix-point equation has a unique solution, and one can check that the distribution given in the lemma is such a solution. The paper~\cite{chamayou-letac} gives several instances of applications of this fruitful method. Another example appears in Section~\ref{sub:LDP}. 

\textit{Proof by time reversal.} 
A more direct approach to Lemma~\ref{lem:CL} is provided by the use of Lemma~\ref{lem:reversal}. 
\begin{figure}[h]
\begin{center}
\includegraphics[height=1.5cm]{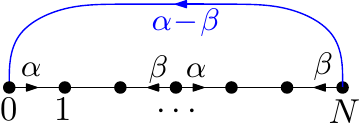}
\end{center}
\caption{Graph $G_N$ for the proof of Lemma~\ref{lem:CL}.}
\label{fig:transience_dir_1D}
\end{figure}

We have, for all $N$, with the graph $G_N$ from Figure~\ref{fig:transience_dir_1D} (which is similar to Figure~\ref{fig:transience_dir} with $d=1$),
\begin{align}
P^{\Z}_{1,\omega}(H_N<H_0) \notag
& = P^{G_N}_{1,\omega}(H_N<H_0) = P^{G_N}_{0,\omega}(H_N<H_0^+)\notag\\
	& =1-P^{G_N}_{0,\omega}(H_0^+<H_N) = 1 - P^{G_N}_{0,\omegach}(H_0^+<H_N)\notag\\
	& =1-\omegach_{0}P^{G_N}_{1,\omegach}(H_0<H_N),\label{eq:cl_reversal}
\end{align}
where the fourth equality comes from applying equation~\eqref{checkw} to each cycle that realizes the event, and noting that the set of those cycles is globally invariant by time inversion. Then, by Lemma~\ref{lem:reversal}, $\omegach_0\sim{\rm Beta(\beta,\alpha-\beta)}$, and $\omegach_1,\ldots,\omegach_{N-1}$ are i.i.d.~with law $\Beta(\beta,\alpha)$ hence
\[\Big(P^{G_N}_{1,\omegach}(H_0<H_N)\text{ under }\P^{(\alpha,\beta)}\Big) \equilaw \Big(P^{\Z}_{1,\omega}(H_0<H_N)\text{ under }\P^{(\beta,\alpha)})\limites{\rm law}{N\to+\infty} 1,\]
by transience to $-\infty$ of the RWDE with parameters $(\beta,\alpha)$. Thus, the last probability in~\eqref{eq:cl_reversal} goes to 1 in probability and we get, by letting $N\to\infty$, 
\[P_{1,\omega}(H_0=+\infty)\sim\Beta(\alpha-\beta,\beta),\]
in accordance to Proposition~\ref{lem:CL} given that the left-hand probability equals $1/R$. 

Let us finally state the annealed scaling limits. Kesten, Kozlov and Spitzer~\cite{kks} proved limit laws for one-dimensional RWRE under general assumptions, however their approach did not give a description of the constants involved. This work was done later. First, the variance in the diffusive case can be computed (cf.~\cite[thm 2.2.1, where $\sigma^2_{P,1}$ should involve $\overline Q$ instead of $Q$]{zeitouni} and references therein). Then, the parameters of the stable limits were obtained in~\cite{limitlaws, stablefluctuations} in terms of Kesten's constant; a fine analysis of the time spent in traps, i.e.\ valleys of the potential, indeed enabled  to relate the tail behaviour of this time to the tail of its quenched average and then to that of the renewal series $R$ (cf.~\cite{renewal} for this crucial part). The explicit value of $C_K(\alpha,\beta)$ deduced from~\cite{chamayou-letac} (cf.~Proposition~\ref{lem:CL} above) is the last ingredient to the following statement: 

\begin{theorem}{($d=1$)}
\label{thm:kks} We have, when $n$ goes to infinity,  
\begin{itemize}
	\item if $0<\alpha-\beta<1$, 
 \begin{align*}
  \frac{X_n}{n^{\alpha-\beta}}
	 &\stackrel{\mathrm{law}}{\longrightarrow} \frac{\sin(\pi(\alpha-\beta))}{2^{\alpha-\beta}\pi}\frac{B(\alpha-\beta,\beta)^2}{\Psi(\alpha)-\Psi(\beta)}\left(\frac1{\mathcal{S}_{\alpha-\beta}^{ca}}\right)^{\alpha-\beta};
\end{align*}
	\item if $\alpha-\beta=1$, for deterministic sequences $(u_n)_n$, $(v_n)_n$ converging to 1, 
 \begin{align*}
  \frac{X_n-v_n\frac{1}{2\beta}\frac{n}{\log n}}{n/(\log n)^2}
    &\stackrel{\mathrm{law}}{\longrightarrow}  \frac{1}{2\beta}\mathcal{S}_{1}^{ca},\label{eqn:thm_x1}
\end{align*}
and in particular,
\begin{align*}
\frac{X_n}{n/\log n}
	 \limites{\rm prob}{}\frac1{2\beta};
\end{align*}
	\item if $1<\alpha-\beta<2$,
 \begin{align*}
  \frac{X_n-\frac{\alpha-\beta-1}{\alpha+\beta-1}n}{n^{\frac{1}{\alpha-\beta}}}
	 &\stackrel{\mathrm{law}}{\longrightarrow} -2\left(-\frac{\pi}{\sin(\pi(\alpha-\beta))}\frac{\Psi(\alpha)-\Psi(\beta)}{B(\alpha-\beta,\beta)^2}\right)^{\frac{1}{\alpha-\beta}} \left(\tfrac{\alpha-\beta-1}{\alpha+\beta-1}\right)^{1+\frac{1}{\alpha-\beta}}\mathcal{S}_{\alpha-\beta}^{ca};
\end{align*}
	\item if $\alpha-\beta>2$, 
\begin{align*}
  \frac{X_n-\frac{\alpha-\beta-1}{\alpha+\beta-1}n}{\sqrt n}
	 &\stackrel{\mathrm{law}}{\longrightarrow}2\left(\frac{\beta(\alpha-1)(\alpha-\beta)}{(\alpha-\beta-2)(\alpha+\beta-1)^2}\right)^{1/2} \mathcal S_2,
\end{align*}
\end{itemize}
where $\Psi$ denotes the classical digamma function,
$\Psi(z)= (\log \Gamma)'(z)=\frac{\Gamma'(z)}{\Gamma(z)},$ 
$\mathcal S_2$ has law $\mathcal N(0,1)$, 
and, for $0<{\kappa_1}<2$, $S_{{\kappa_1}}^{ca}$ is a totally asymmetric stable random variable of parameter ${\kappa_1}$, such that 
\[E[e^{it\mathcal S_{{\kappa_1}}^{ca}}]=\begin{cases} 
e^{-(-it)^{\kappa_1}}&\text{if $0<{\kappa_1}<1$}\\
e^{-\frac\pi2|t|-i|t|\ln|t|}&\text{if ${\kappa_1}=1$}\\
e^{(-it)^{\kappa_1}}&\text{if $1<{\kappa_1}<2$.}
\end{cases}\]
\end{theorem}
\begin{remark}
Note that RWDE is the only example where the constants in these limit theorems are fully explicit.
\end{remark}
\begin{remark} In the case ${\kappa_1}=2$, a CLT with scaling $\sqrt{n\ln n}$ holds by~\cite{kks}, although up to our knowledge the explicit constant has not been rigorously computed so far. 
\end{remark}

\subsection{Exact expression of the large deviation rate function}\label{sub:LDP}

Consider the random walk in beta random environment on $\Z$ with parameters $(\alpha, \beta)$. 
Assume that $\alpha>\beta$ so that the RWDE is transient in the positive direction. Consider the stopping times for $k\in \N$,
$$
H_k=\inf\{n\ge 0, \; X_n=k\}.
$$
Define the function
\begin{eqnarray}\label{phi}
\phi(\w,\lambda)=E_{0,\w}\left[ \lambda^{H_1}\right].
\end{eqnarray}
The sequence ${1\over k}H_k$ satisfies a large deviation principle
with rate function given by the Legendre transform of $\lambda\mapsto \E\left[ \log\phi(\w,\lambda)\right]$, cf.~\cite{denhollander},
Lemma~VII.6 (and previously \cite{greven-denhollander,comets-gantert-zeitouni}).
It is also remarked in Lemma~VII.12, that the function $\phi(\w,\lambda)$ satisfies a recursion equation, and hence can be
represented as a continued fraction. 
What we show below is that in
case of Beta environments, the law of $\phi(\w,\lambda)$ is explicit, hence the rate function can be expressed as the Legendre
transform of a simple integral.

For $z\in (-\infty, 1)$, we introduce the hypergeometric density $h^{(1)}(\alpha, \beta; z)(\d u)$, cf.~\cite{chamayou-letac} example 6 page 13, 
which is the density on the interval $(0,1)$ given by
$$
 {\Gamma(\alpha+\beta)\over \Gamma(\alpha)\Gamma(\beta)}{1\over F(\alpha, \alpha; \alpha+ \beta; \lambda^2)} 
 u^{\alpha-1}(1-u)^{\beta-1} (1-u z)^{-\alpha} \indic_{(0,1)}(u) \d u,
 $$
 where 
 $$
 F(a,b;c;z)=\;  _2F_1(a,b;c;z)={\Gamma(c)\over \Gamma(b)\Gamma(c-b)} \int_{0}^1 
 u^{b-1}(1-u)^{c-b-1} (1-u z)^{-a} \d u
 $$ 
 is the classical hypergeometric function.
\begin{theorem}
\label{lawZ}
Let $\lambda\in [0,1]$. Consider the random variable
$$
Z(\w)={1\over \lambda} \phi(\w,\lambda).
$$
 Then $Z$ follows the hypergeometric law $h^{(1)}(\alpha, \beta; \lambda^2)$.
  \end{theorem}
\begin{proof}
  We start from the following elementary lemma. This identity comes from \cite{denhollander}, Lemma~VII.12\footnote{The first author 
  thanks Alejandro Ram\'irez for mentioning this identity and its relation with large deviation.}.
  \begin{lemma}
 The random variable $Z$ satisfies a distributional equation:
  \begin{equation}
 \label{distrib-equ}
 Z\; \equilaw \; {Y\over 1+Y-\lambda^2 Z},
 \end{equation}
 where on the right hand side, $Y$ is a random variable independent of $Z$ and with distribution
 $$
 Y\; \equilaw \; {U\over 1-U}, \;\;\; \hbox{ $U$ is beta$(\alpha,\beta)$ distributed}.
 $$
 \end{lemma}
 \begin{proof}
 Indeed, the following identity
 $$
 E^\w_0\left[\lambda^{\tau_1}\right]= \lambda \w_0 +(1-\w_0)E_{0}^{\theta_{-1}\w}\left[\lambda^{\tau_1}\right] E^\w_0\left[\lambda^{\tau_1}\right],
 $$
 where $\theta_{-1} \w$ is the environment shifted one step to the left, is an obvious consequence of Markov property.
 Since $\w_0$ and
 $E_{0}^{\theta_{-1}\w}\left[\lambda^{\tau_1}\right]$
 are independent, the relation (\ref{distrib-equ}) follows easily.
\end{proof}

There exists a unique distribution satisfying equation (\ref{distrib-equ}). Indeed, from \cite{denhollander}, Lemma VII.12, we can deduce from
Lemma~\ref{distrib-equ} that $\phi(\w,\lambda)$ can be represented as a converging continued fraction.

We make the change of variable to the random variable $X$ on $(0,+\infty)$ given by
$$
X={Z\over 1-Z}.
$$
Simple computation implies that $X$ must satisfy the distributional identity
  \begin{equation}
 \label{distrib-equ2}
 X\; \equilaw \; Y{1+X\over 1+(1-\lambda^2) X},
 \end{equation}
 where on the right hand side $Y$ is independent of $X$ with same distribution
as in the statement of Theorem \ref{lawZ}.
We first prove that if $X$ follows the distribution $h^{(2)}(\alpha,\beta;\lambda^2)(\d v)$ where $h^{(2)}(\alpha,\beta;z)(\d v)$ is the distribution
on $(0,+\infty)$ given by
$$
 {\Gamma(\alpha+\beta)\over \Gamma(\alpha)\Gamma(\beta)}{1\over F(\alpha, \alpha; \alpha+ \beta; \lambda^2)} 
 v^{\alpha-1}(1+v)^{-\beta} (1+(1-z)v)^{-\alpha} \indic_{\R_+}(v) \d v,
$$
then $X$ is solution of the distributional equation (\ref{distrib-equ2}). This is inspired by the type of explicit solutions that
appears in \cite{chamayou-letac}, Section 5.3, even if it does not seems to enter in one of the example treated in this paper
(even after change of variables). 

Let us assume that $X$ follows $h^{(2)}(\alpha, \beta; \lambda^2)$ and compute the $s$-moments, $s>0$, of both sides of (\ref{distrib-equ2}). The $s$-moment of the left hand side is
\begin{eqnarray*}
\E[X^s] &= &
{\Gamma(\alpha+\beta)\over \Gamma(\alpha)\Gamma(\beta)}{1\over F(\alpha, \alpha; \alpha+ \beta; \lambda^2)} 
\int_0^\infty  v^{\alpha+s-1}(1+v)^{-\beta} (1+(1-\lambda^2)v)^{-\alpha} \indic_{\R_+}(v) \d v
\\
&=&
{\Gamma(\alpha+\beta)\over \Gamma(\alpha)\Gamma(\beta)}{F(\alpha,\alpha+s; \alpha+\beta; \lambda^2)\over F(\alpha, \alpha; \alpha+ \beta; \lambda^2)} 
{\Gamma(\alpha+s) \Gamma(\beta-s)\over\Gamma(\alpha+\beta)}
\\
&=&
{\Gamma(\alpha+s)\Gamma(\beta-s)\over \Gamma(\alpha)\Gamma(\beta)}
{F(\alpha,\alpha+s; \alpha+\beta; \lambda^2)\over F(\alpha, \alpha; \alpha+ \beta; \lambda^2)},
\end{eqnarray*}
while, by independence, the $s$-moment of the right hand side of (\ref{distrib-equ2}) equals
\begin{eqnarray*}
\E[Y^s]\E\left[{(1+X)^s\over (1+(1-\lambda^2)X)^2}\right],
\end{eqnarray*}
and we have
$$
\E[Y^s]= 
{\Gamma(\alpha+s)\Gamma(\beta-s)\over \Gamma(\alpha)\Gamma(\beta)}
$$
and
\begin{eqnarray*}
&&\E\left[{(1+X)^s\over (1+(1-z)X)^2}\right]
\\
&=&
{\Gamma(\alpha+\beta)\over \Gamma(\alpha)\Gamma(\beta)}{1\over F(\alpha, \alpha; \alpha+ \beta; \lambda^2)} 
 \int_0^\infty v^{\alpha-1}(1+v)^{-\beta+s} (1+(1-\lambda^2)v)^{-\alpha-s} \indic_{\R_+}(v) \d v
\\
&=&
{F(\alpha+s, \alpha; \alpha+\beta; \lambda^2)\over F(\alpha, \alpha; \alpha+ \beta; \lambda^2)}.
\end{eqnarray*}
Using the property $F(a,b;c;z)=F(b,a;c;z)$, we get that the $s$-moments of both sides of (\ref{distrib-equ2}) coincides.
Simple computation shows that it implies that $Z$ is solution of the distributional identity (\ref{distrib-equ}) if it follows the law
$h^{(1)}(\alpha,\beta;\lambda^2)(\d v)$. The proposition follows from the uniqueness of the solution of this identity.

\end{proof}

\bibliographystyle{acm}
\bibliography{biblio}

\end{document}